\documentclass[12pt,leqno]
{amsart}

\usepackage[margin=1.2in]{geometry}

\usepackage[colorlinks, linkcolor=blue,citecolor=blue]{hyperref}

\usepackage{amsmath,epsfig,graphicx,color, float}
\usepackage{subcaption}
\usepackage{relsize}
\usepackage{comment}

\usepackage{enumitem}

\usepackage{appendix}

\numberwithin{table}{section}
\numberwithin{figure}{section}
\numberwithin{equation}{section}

\definecolor{darkblue}{rgb}{.2, 0.2,.8}
\definecolor{darkgreen}{rgb}{0,0.5,0.3}
\definecolor{darkred}{rgb}{.8, .1,.1}
\newcommand{\red}{\color{darkred}}
\newcommand{\blue}{\color{darkblue}}

\newcommand{\bfY}{\vect{Y}}
\newcommand{\bfy}{\vect{y}}

\newcommand{\bfX}{\vect{X}}

\newcommand{\bfalp}{\vect{\alpha}}

\newcommand{\bfpi}{\vect{\pi}}
\newcommand{\bfT}{\mat{T}}
\newcommand{\bft}{\vect{t}}
\newcommand{\bfe}{\vect{e}}

\newcommand{\bfI}{\mat{I}}

\newcommand{\bfbeta}{\vect{\beta}}
\newcommand{\bfTheta}{\vect{\Theta}}
\newcommand{\bftheta}{\vect{\theta}}

\newcommand{\0}{\mat{0}}

\newcommand{\R}{\mathbb{R}}
\newcommand{\E}{\mathbb{E}}
\renewcommand{\P }{{\mathbb P}}

\newcommand{\eqd}{\stackrel{d}{=}}

\newtheorem{lemma}{Lemma}[section]

\newtheorem{proposition}[lemma]{Proposition}
\newtheorem{definition}[lemma]{Definition}

\newtheorem{example}[lemma]{Example}

\newtheorem{remark}{Remark}[section]

\usepackage{bm}
\newcommand{\vect}[1]{\pmb{#1}}
\newcommand{\mat}[1]{\boldsymbol{\bm #1}}

\DeclareMathOperator*{\argmax}{arg\,max}

\allowdisplaybreaks

\usepackage[section]{algorithm}
\usepackage{algorithmicx}
\usepackage{algpseudocode}

\usepackage{makecell}

\begin{document}
\title[Heavy-tailed phase-type distributions]{Heavy-tailed phase-type distributions: a unified approach}

\author[M. Bladt]{Martin Bladt}
\address{Faculty of Business and Economics,
University of Lausanne,
Quartier de Chambronne,
1015 Lausanne,
Switzerland}
\email{martin.bladt@unil.ch}

\author[J. Yslas]{Jorge Yslas}
\address{Institute of Mathematical Statistics and Actuarial Science,
University of Bern,
Alpeneggstrasse 22,
CH-3012 Bern,
Switzerland}
\email{jorge.yslas@stat.unibe.ch}

\begin{abstract}
{\color{black}A phase-type distribution is the distribution of the time until absorption in a finite state-space time-homogeneous Markov jump process, with one absorbing state and the rest being transient. These distributions are mathematically tractable and conceptually attractive to model physical phenomena due to their interpretation in terms of a hidden Markov structure.} Three recent extensions of regular phase-type distributions give rise to models which allow for heavy tails: discrete- or continuous-scaling; fractional-time semi-Markov extensions; and inhomogeneous time-change of the underlying Markov process. In this paper, we present a unifying theory for heavy-tailed phase-type distributions for which all three approaches are particular cases. Our main objective is to provide useful models for heavy-tailed phase-type distributions, but any other tail behavior is also captured by our specification. We provide relevant new examples and also show how existing approaches are naturally embedded. Subsequently, two multivariate extensions are presented, inspired by the univariate construction which can be considered as a matrix version of a frailty model. We provide fully explicit EM-algorithms for all models and illustrate them using synthetic and real-life data. 
\end{abstract}
\maketitle

\section{Introduction}

{\color{black}Phase-type (PH)} distributions have been employed extensively in applied probability since they often provide exact and explicit solutions to complex stochastic problems. Another attractive property of PH distributions is that they form a dense class in the set of distributions in the positive half-line in the sense of weak convergence (see \cite[Section 3.2.1]{Bladt2017}). However, and despite their denseness, PH distributions are always light-tailed, which may be a problem when heavy tails are present.

 At least three approaches to remedy this problem have been introduced in the literature. The first one, originally introduced in \cite{bladt2015calculation} and called the NPH class of distributions, consists of considering PH distributions scaled by nonnegative discrete random variables{\color{black}, $N$}. 
 This construction principle has the advantage that the resulting distribution maintains the interpretation as being the absorption time of a homogeneous Markov jump process but in an infinite-dimensional state-space. This, indeed, allows for genuinely heavy tails for the resulting distribution. {\color{black} For instance, in \cite{rojas2018asymptotic}, the authors showed that if the scaling component is unbounded (but otherwise arbitrary), then the resulting distribution is always heavy-tailed in terms of non-existent moment generating functions (see also \cite{su2006behavior} for more general results)}.  However, their different functionals are in terms of infinite-dimensional matrices, which in practice, can only be computed up to a finite number of terms. More recently, in \cite{albrecher2021cph}, the authors considered continuous scaling and showed that closed-form expressions for different functionals of the resulting distributions can be obtained. They denoted this class by CPH. Another advantage of continuous scaling is that it maintains the (finite) dimension of the underlying PH.
 
 A second approach was introduced in \cite{albrecher2019mml} by considering a time-fractional version of the underlying stochastic process dynamics, effectively moving into the semi-Markov domain. Together with subsequent multivariate extensions based on rewards (cf. \cite{albrecher2020mulmml,albrecher2020mulfrac}), these models were shown to be feasible models for applications such as non-life insurance modeling. More recently, \cite{ilibladt} showed that these models are relevant in describing lifetimes and performing the corresponding life-insurance calculations.
 
The third approach, introduced in \cite{albrecher2019inhomogeneous}, consists of allowing the Markov jump process to be time-inhomogeneous in the construction principle of PH distributions leading to the class of inhomogeneous phase-type (IPH) distributions. An advantage of this approach is that one gains substantial flexibility on the tails: not only are heavy tails possible but also, e.g., lighter tails than exponential-decay can be obtained. Further extensions to covariate-dependent distributions can be found in \cite{bladt2020survival}, which is particularly well-suited for survival analysis applications.

 Estimation of PH distributions was initially developed to calibrate such stochastic models to real-life data, and it is a well-developed topic in the literature. It is typically done via an expectation-maximization (EM) algorithm (\cite{asmussen1996fitting}), although other methods such as an MCMC approach have been introduced (\cite{bladt2003estimation}). More recent trends have moved towards considering PH-based models purely as flexible models for statistical fitting, irrespectively of their explicit and closed-form formulas. This data-driven approach is particularly attractive compared to other classical alternatives (for instance, kernel smoothing) since there is the implicit interpretation of an underlying process traversing through different states before it terminates, which is easy to justify in many application areas. Algorithms for discretely-scaled PH distributions, IPH models, and continuously-scaled PH distributions can be found, respectively, in  \cite{bladt2017fitting}, \cite{albrecher2020fitting}, and \cite{albrecher2021cph}. To the best of the authors' knowledge, an EM-based estimation procedure for fractional phase-type distributions (also called matrix Mittag-Leffler distributions) has not been considered before the present work, with \cite{albrecher2019mml} performing a purely numerical multi-dimensional maximum-likelihood estimation.

The primary purpose of this paper is to present a unified theory that englobes the above approaches to produce heavy-tailed phase-type distributions. The construction principle of the proposed models is simple to conceptualize and can be seen as a matrix extension of the frailty model in survival analysis.
However, the flexibility of the underlying Markov structure allows for very different objects to be constructed as special cases. More precisely, we study IPH distributions with intensity matrices scaled by any nonnegative random variable. In other words, we impose both a random and a deterministic component which modify the speed at which the finite state-space is traversed by the Markov process, such that absorption times can possess any desired tail and body behavior, in particular obtaining heavy-tailed distributions. Inhomogeneous generalizations of \cite{albrecher2021cph,rojas2018asymptotic}, the matrix Mittag-Leffler models of \cite{albrecher2019mml}, and randomly scaled generalizations of \cite{albrecher2019inhomogeneous,bladt2020survival} (with the possibility of missing covariates) are all comprised in this rich class.

{\color{black}In terms of physical interpretation, the latent variables play different roles. The underlying Markov dynamics aim to model heterogeneity by assuming that unobserved traversing of states has occurred. In contrast, the interpretation of the scaling component is closely related to the statistical concept of frailty.
Recall that frailty models (see, e.g., \cite{wienke2010frailty} for a comprehensive account of such models) specify a multiplicative random effect on the hazard rate of a distribution, effectively accounting for unobserved covariates in a Cox proportional hazards model. 
In contrast, we specify a multiplicative random effect on the intensity function of a Markov jump process. Nonetheless, since for IPH distributions, the hazard rate and intensity function are asymptotically equivalent (cf. \cite{bladt2020survival}), the scaling variable can also be interpreted as accounting for heterogeneity or missing covariates in an asymptotically proportional hazards model.}

The secondary aim of the paper is to present multivariate models based on this construction, which can be interpreted as generalizations of the shared and correlated frailty models (cf. \cite{wienke2010frailty}). We derive EM algorithms for maximum-likelihood estimation of all the proposed models, which can be implemented either in full generality or by simplifying some assumptions and tailoring the methods for the specific application. For pedagogical reasons, we build up the multivariate case from the univariate one, although a top-bottom approach is also possible.

The rest of the paper is organized as follows. In Section~\ref{sec:iph}, we present an overview of the class of IPH distributions and some important properties for our present purposes. 
In Section~\ref{sec:siph}, we introduce our main univariate model, which we call scaled inhomogeneous phase-type, derive its main properties, give several parametric examples relevant for real-life applications, and propose a generalized EM algorithm for its maximum-likelihood estimation.
In Section~\ref{sec:mult}, we present a multivariate extension inspired by the shared frailty model and show how estimation of the proposed models can be performed via EM algorithms. 
In Section~\ref{sec:correlated}, we present a different multivariate extension, now based on the construction principle of correlated frailty models, and derive an EM algorithm for maximum-likelihood estimation.
In Section~\ref{sec:Numexamples}, we present several numerical illustrations. 
Finally, Section~\ref{sec:conclusion} concludes. 

\section{Preliminaries}\label{sec:iph}
This section presents the relevant preliminaries on time-inhomogeneous Markov jump-processes and their absorption times. The distributions of the latter times are the building blocks for the scaled models introduced in Section~\ref{sec:siph}. For distributional equality between two random variables $X,Y$, we use the notation $X\stackrel{d}{=}Y$, while the notation $X\sim F$ for $F$ a distribution function, density, or acronym is understood as $X$ following the distribution uniquely associated with $F$. Unless stated otherwise, equalities between random objects hold almost surely. For two real-valued functions, $g,h$ the terminology $g(t)\sim h(t)$, as $t\to a\in\mathbb{R}\cup\{-\infty,+\infty\}$ is defined as $\lim_{t\to a}g(t)/h(t)=1$. If $a$ is not explicitly mentioned, it is assumed to be $+\infty$.

Let $ ( X_t )_{t \geq 0}$ denote a time-inhomogeneous Markov jump process on the state-space $E = \{1, \dots, p, p+1\}$, where states $1,\dots,p$ are transient and state $p+1$ is absorbing. In this way, $ ( X_t)_{t \geq 0}$ has an intensity matrix of the form
\begin{align*}
	\mat{\Lambda}(t)= \left( \begin{array}{cc}
		\bfT(t) &  \bft(t) \\
		\0 & 0
	\end{array} \right)\,, \quad t\geq0\,.
\end{align*}
%
Since $\mat{\Lambda}(t) $ is an intensity matrix, the sum of its rows is zero for any time $t\geq0$, and so the identity $\bft (t)=- \bfT(t) \, \bfe,$ holds, where $\bfe $ is the $p$--dimensional column vector of ones. 
Moreover, the probability transition matrix $\boldsymbol{P}(s, t) = \{p_{k,l}(s,t)\}_{k,l \in E}$ of $ ( X_t )_{t \geq 0}$, where 
\begin{align*}
	p_{k,l}(s,t) = \mathbb{P}(X_{t}=l \mid X_{s}=k) \,, \quad k,l \in E \,,
\end{align*}
is given in terms of the product integral (see \cite{albrecher2019inhomogeneous})
\begin{align*}
	 \mat{P} (s, t) = \prod_s^t\left(\bfI + \mat{\Lambda} (u) du \right) =\begin{pmatrix}
    \prod_s^t\left(\bfI + \bfT (u) du \right) & \bfe-\prod_s^t\left(\bfI + \bfT (u) du \right) \bfe \\
    \textbf{0} & 1
    \end{pmatrix} \,.
\end{align*}

To avoid degeneracies, we assume that the process starts almost surely in a non-absorbing state $k\le p$ with probabilities given by $ \pi_{k} = \P(X_0 = k)$, $k = 1,\dots, p$. In vector notation, we write $\bfpi = (\pi_1 ,\dots,\pi_p )$. In the sequel, we follow the convention that greek boldface lowercase letters are row-vectors, while roman boldface lowercase letters are column-vectors.
Thus $\sum_{k=1}^p\pi_k=\bfpi \bfe=1$.

The main quantity of interest of such a process for our present purposes is the time taken to reach the absorbing state, denoted by
\begin{align*}
	\tau = \inf \{ t \geq  0 \mid X_t = p+1 \}\,,
\end{align*}
which has an inhomogeneous phase-type distribution (cf.\ \cite{albrecher2019inhomogeneous}) with representation $(\bfpi,\bfT(t))$, and we write $\tau \sim \mbox{IPH}(\bfpi,\bfT(t) )$. Application of such random variables to statistical modeling is often treated for the special case $\bfT(t) = \lambda(t)\,\bfT$, with $\lambda(t)$ some known nonnegative real function, known as the intensity function, and $\bfT$ a fixed sub-intensity matrix. We adopt this approach in the present text. Thus we may simply write $\tau \sim  \mbox{IPH}(\bfpi , \bfT , \lambda )$. The interested reader is referred to  \cite{Bladt2017} for a comprehensive account of the $\lambda\equiv1$ case and \cite{albrecher2019inhomogeneous} for further reading on general IPH distributions. 

The restricted class of IPH distributions is nonetheless quite versatile. Whenever $Y  \sim  \mbox{IPH}(\bfpi , \bfT , \lambda )$, then there exists a function $h$ such that \begin{equation}\label{gtrans}
Y \eqd h(Z) \,,
\end{equation}
where $Z \sim \mbox{PH}(\bfpi , \bfT )$.
More specifically, the relationship between $h$ and $\lambda$ is given by 
\begin{equation*}
h^{-1}(t) = \int_0^t \lambda (t)dt, \quad t\ge0,   \label{eq:transformation-g}
\end{equation*}
or in terms of derivatives
\begin{equation*}
\lambda (t) = \frac{d}{dt}h^{-1}(t) \,.  \label{eq:transformation-lambda} 
\end{equation*}
To make sure that $Y$ is positive, unbounded, and almost surely finite, we require that $$h^{-1}(t)<\infty\,,\quad\forall t>0\,,\quad \lim_{t\uparrow\infty}h^{-1}(t)=\infty\,.$$
The density $f_Y$ and survival function $S_Y$ of $Y \sim  \mbox{IPH}(\bfpi , \bfT , \lambda )$ are explicit in terms of matrix exponential formulas, and given by
\begin{eqnarray*}
 f_Y(y) &=& \lambda (y)\, \vect{\pi}\exp \left( \int_0^y \lambda (t)dt\ \mat{T} \right)\vect{t}, \quad y\ge0, \label{eq:dens-IPH} \\
 S_Y(y) &=&  \vect{\pi}\exp \left( \int_0^y \lambda (t)dt\ \mat{T} \right)\vect{e}, \quad y\ge0.  \label{eq:cdf-IPH}
\end{eqnarray*}

%
The tail behavior of IPH distributions is driven by the asymptotic behavior of the $\lambda$ function. Table~\ref{tab:trans} presents an overview of some commonly used intensities and transforms for generating parametric IPH distributions (see  \cite{bladt2021matrixdist}). Applications and estimation can be found, for instance, in \cite{albrecher2019inhomogeneous, albrecher2020fitting, bladt2020survival}. Their names are inspired by the $p=1$ case, e.g., a matrix-Weibull distribution reduces to the regular Weibull distribution when $\bfT$ is a $1\times 1$ matrix. In general, the additional parameters allow for more flexible modeling in the body of the distribution while preserving the same tail behavior as the scalar case. 

\begin{table}[h]
\centering
\begin{tabular}{lccc}
  \hline
      Distribution              &   $\lambda(t)$       & $h(z)$ & Parameters Domain  \\
  \hline
Matrix-Pareto       & $( t + \eta )^{-1}$ & $ \eta \left( \exp(z)-1\right)$ & $\eta>0$ \\[2mm] 
Matrix-Weibull      & $\eta t ^{\eta -1}$    & $z^{1/\eta}$ & $\eta>0$ \\ [2mm]
Matrix-Lognormal    & ${\gamma (\log(t+1))^{\gamma-1}}/{(t+1)} $    & $\exp(z^{1/\gamma})-1$      & $\gamma > 1$     \\ [2mm]
Matrix-Loglogistic  & $\eta t^{\eta -1} / (t^\eta + \gamma^{\eta})$    & $\gamma (\exp(z)-1)^{1/\eta}$    & $\gamma,\eta >0$     \\ [2mm]
Matrix-Gompertz     & $\exp (\eta t)$    & $\log( \eta z  + 1 ) / \eta$      & $\eta > 0$     \\ 
   \hline
\end{tabular}
\caption{Some IPH distributions with their respective intensities and transforms.}
\label{tab:trans}
\end{table}

\section{Scaled inhomogeneous phase-type distributions}\label{sec:siph}
In this section, we introduce the main general specification of the paper and then derive some special cases together with a detailed analysis of their specific tail asymptotics. The central assumption underpinning our model is that an individual's intensity function depends on an unobservable nonnegative random variable $\Theta$. More specifically, we focus on the case where $\Theta$ acts multiplicatively on the intensity function, that is 
 \begin{align}\label{eq:scaledint}
	 	\lambda(t ; \Theta ) = \Theta \lambda(t), \quad t\ge0,
	\end{align}
where $\lambda$ is the baseline intensity function. If we denote by $Y$ a random variable with intensity \eqref{eq:scaledint}, then we have that
\begin{align}\label{eq:cond}
	Y \mid \Theta = \theta \sim \mbox{IPH} (\bfpi, \bfT, \theta \lambda ) \,.
\end{align}
For the representation of these distributions, we make use of functional calculus. More specifically, if $g$ is an analytic function and $\mat{A}$ is a matrix, we can express  $g(\mat{A})$ by Cauchy's formula 
\begin{align*}
	g(\mat{A}) = \frac{1}{2 \pi i} \oint_{\Gamma} g(z)(z\bfI - \mat{A}) dz \,,
\end{align*}
where $\Gamma$ is the simple closed  path in $\mathbb{C}$ which encloses the eigenvalues of $\mat{A}$ (cf. \cite[Section 3.4.]{Bladt2017} for details). 

The following result characterizes the density and survival functions of $Y$. In particular, observe that the asymptotic behavior of the tail of $Y$ depends on both the shape of $\mathcal{L}_\Theta$, the Laplace transform of $\Theta$, and on $\lambda$.  In subsection~\ref{sec:example}, we give an in-depth  asymptotic analysis of the new parametric models presented in this paper.

 \begin{proposition}
 	Let $Y$ be given by \eqref{eq:cond}. Then we have that, for $y\ge0$,
 	\begin{enumerate}[label=(\arabic*)]
	  \setlength\itemsep{0.8em}
 		\item $S_{Y}(y) = \bfpi \mathcal{L}_\Theta ( - h^{-1}(y) \bfT ) \bfe $,
 		\item $f_Y(y) =  -\lambda(y) \bfpi \mathcal{L}_\Theta^{\prime} ( - h^{-1}(y) \bfT ) \bft$,
 	\end{enumerate}
 	where $h^{-1}(y)= \int_{0}^{y}\lambda(t)dt$. 
 \end{proposition}
 \begin{proof}
 	Property (1) follows from
  \begin{align*}
 	S_{Y}(y) & = \int \bfpi \exp({ \theta h^{-1}(y) \bfT   }) \bfe \, dF_\Theta(\theta) \\
 	& =  \bfpi \int \exp({ \theta h^{-1}(y) \bfT   })  \, dF_\Theta(\theta) \bfe \\
	&=\bfpi \int \frac{1}{2 \pi i} \oint_{\Gamma} \exp(z)(z\bfI - \theta h^{-1}(y) \bfT) dz  \, dF_\Theta(\theta) \bfe \\
        &=\bfpi \frac{1}{2 \pi i} \oint_{\Gamma}\int  \exp(z)(z\bfI - \theta h^{-1}(y) \bfT)  dF_\Theta(\theta)  \, dz\,\bfe \\
 	& = \bfpi \mathcal{L}_\Theta ( - h^{-1}(y) \bfT ) \bfe \,,
 \end{align*}
where we have used functional calculus to define the Laplace transform evaluated at a matrix. 
Taking derivatives in the expression above yields
 \begin{align*}
 f_{Y}(y) & = -\bfpi \int \theta \lambda(y) \bfT \exp({ \theta h^{-1}(y) \bfT   })\, dF_\Theta(\theta) \bfe 
 \end{align*}
 from which (2) follows. 
 \end{proof}
 
 The following lemma shows that $Y$ has the same distribution as the transformation of a scaled PH distribution. Such a representation is useful for simulation and for estimation, as is apparent in later sections.
 \begin{lemma}
 	Let $Y$ be given in terms of \eqref{eq:cond}. Then, $Y \eqd h(Z /\Theta)$, where $Z \sim \mbox{PH}(\bfpi, \bfT)$, independent of $\Theta$, and $h^{-1}(y)= \int_{0}^{y}\lambda(t)dt$. 
 \end{lemma}
 \begin{proof}
 	\begin{align*}
 		\P (h(Z /\Theta) > y) &= \int \P (h(Z /\theta) > y \mid \Theta=\theta)   dF_{\Theta}(\theta) \\
 		&= \int \P ( Z  > \theta h^{-1}(y) \mid \Theta=\theta)   dF_{\Theta}(\theta) \\
 		& = \int \bfpi \exp({ \theta h^{-1}(y) \bfT   }) \bfe \, dF_\Theta(\theta) \\
 		& = S_{Y}(y) \,.
 	\end{align*}
 	
 \end{proof}
We now make the following formal definition of a random variable $Y$ satisfying \eqref{eq:cond}.
 
 \begin{definition}\rm 
 	A random variable $Y$ is said to have {\em scaled inhomogeneous phase-type distribution} (SIPH) with representation $(\bfpi, \bfT, \lambda)$ and scaling distribution $F_\Theta$ if its survival function is given by 
 	\begin{align*}
 		S_Y(y) = \bfpi \mathcal{L}_\Theta \left( - \int_0^y \lambda(t) dt \,\bfT \right) \bfe, \quad y\ge0.
 	\end{align*}
 	We write $\mbox{SIPH}(\bfpi, \bfT, \lambda, \Theta)$.
 \end{definition}

 \begin{remark}[Existing special cases of heavy-tailed PH models] \rm 
 	
	i) For $\lambda \equiv 1$ and $\Theta\in \mathbb{N}$, almost surely, we obtain the class of NPH distributions introduced in \cite{bladt2015calculation}, while for $\lambda \equiv 1$ and $\Theta\in \mathbb{R}_+$, almost surely, we recover the CPH class in \cite{ albrecher2021cph, rojas2018asymptotic}. 
 
 	ii) Consider {\color{black} a Matrix Mittag Leffler (fractional phase-type) random variable }$Y \sim \mbox{MML}(\alpha, \bfpi, \bfT)$ as introduced in \cite{albrecher2019mml}. Then, it can be shown that 
 	\begin{align*}
 		Y \eqd Z^{1/\alpha} S_{\alpha} = (Z S_{\alpha}^{\alpha} )^{1/\alpha} \,,
 	\end{align*}
 	where $Z \sim \mbox{PH}(\bfpi, \bfT)$ and $S_{\alpha}$ is an independent (positive stable) random variable with Laplace transform given by $\exp(-u^{\alpha})$, $\alpha \in (0,1]$. Hence, we have that $Y $ is SIPH distributed with $h(x) = x^{1/\alpha}$ and $\Theta = 1/S_{\alpha}^{\alpha} $. This class of distributions is the time-fractional counterpart of PH distributions and can be interpreted as absorption times of a stochastic process that traverses through a finite number of states. The holding times of the latter are Mittag-Leffler distributed, which are regularly varying, and thus can possess abnormally large holding times compared to a Markov framework. However, the boundary case $\alpha=1$ corresponds to the usual exponential holding times, and thus there is a regime-shift with respect to tail behavior.
	
	 	iii) When the scaling component $\Theta$ degenerates to a point $\Theta \equiv k\in\mathbb{R}_+$, we recover the class of IPH distributions. This also implies that the class of SIPH distributions, with a given and fixed intensity, is dense in the class of distributions on the positive real line. The argument is omitted, but it is a simple application of convergence through the diagonal of an array, for instance, by choosing a sequence of scalings $\Theta_n$ with constant mean $k$ and variances shrinking to zero.

 \end{remark}
 
\begin{remark}[Frailty models as a $p=1$ special case]\rm

Recall that for a continuous and positive random variable $Y$, the hazard function $\mu_Y$ is given by 
\begin{align*}
	\mu_Y(t) = \frac{f_Y(t)}{S_Y(t)}, \quad t\ge0.
\end{align*}
Sometimes, it is convenient to deal with the cumulative or integrated hazard function  $M_Y$, which is given by 
\begin{align*}
	M_Y(t) = \int_0^{t} \mu_Y (s) ds = -\log (S_Y(t)), \quad t\ge0.
\end{align*}

The classical frailty model in survival analysis assumes that the hazard function of an individual depends on an unobservable random variable $\Theta$. More specifically, it assumes that $\Theta$ acts multiplicatively on a baseline hazard function $\mu$, that is
\begin{align}\label{eq:frailty}
	\mu(t; \Theta) = \Theta \mu(t), \quad t\ge0.
\end{align}
Here, the random variable $\Theta$ is known as the frailty. If we denote by $Y$ the random variable with the above hazard, then the survival function of $Y \mid \Theta = \theta$ is given by 
\begin{align*}
	S_{Y|\Theta}(y | \theta) = \exp \left( -\theta \int_0^y \mu(t)dt \right) = \exp \left( -\theta M(y)\right) \,.
\end{align*}
Thus, the unconditional survival function of $Y$ is given by
\begin{align*}
	S_Y(y) = \int_0^\infty S_{Y|\Theta}(y | \theta) dF_\Theta(\theta) = \int_0^\infty \exp \left( -\theta M(y)\right) dF_\Theta(\theta)  = \mathcal{L}_\Theta(M(y)) \,.
\end{align*}
Furthermore, model \eqref{eq:frailty} can incorporate covariates $\bfX = (X_1,\dots, X_{q})^{\top}\in\mathbb{R}^q$ in a similar way to the Cox's proportional hazards model via 
\begin{align*}
	\mu(t; \Theta, \bfX) = \Theta \mu(t) \exp ( \bfbeta \bfX ), \quad t\ge0,
\end{align*}
where $\bfbeta\in\mathbb{R}^q$ is a $q$-dimensional parameter row vector. 
 Note that when the frailty degenerates to $\Theta\equiv1$, one recovers the proportional hazards model, meaning that the frailty model generalizes the proportional hazards model. 
Commonly employed distributions as frailties include the Gamma and the positive stable distributions, among others. 

In \cite{bladt2020survival}, it was shown that the intensity function of an IPH distribution is asymptotically equivalent to its hazard function. More specifically, we have that $\lambda(t) \sim C \mu(t)$ as $t \to \infty$ with $C>0,$ a positive constant. In particular, when $p = 1$, the previous asymptotic result becomes equality. It follows that the frailty model is a special case of our more general matrix specification of SIPH distributions, when $p=1$.

\end{remark}

\begin{remark}[Incorporating regressors] \rm

	As in the frailty model, we can introduce covariates into \eqref{eq:scaledint} via 
 \begin{align*}
	 	\lambda(t ; \Theta, \bfX) = \Theta \lambda(t) \exp( \vect{\beta} \bfX),\quad t\ge0.
	\end{align*}
In this case, we write $Y \sim \mbox{SIPH}(\bfpi, \bfT, \lambda, \Theta, \bfbeta)$ to denote a random variable with above intensity. Note that the proportional intensities model introduced in \cite{bladt2020survival} is retrieved if the scaling distribution degenerates to $\Theta\equiv1$ for all individuals.
Consequently, the SIPH model is a generalization of
the proportional intensities model. 
\end{remark}

In what follows, we mostly restrict ourselves to the model \eqref{eq:scaledint} without covariates, the extension being straightforward but somewhat distracting to the current train of thought. Moreover, we assume that $\Theta$ is a continuous random variable unless stated otherwise.

\subsection{Novel examples}\label{sec:example}

Next, we present a suite of new examples that arise naturally as matrix extensions of some well-known frailty models, providing along the way some insight into the precise asymptotic behavior of the proposed models. In Appendix \ref{ap:def}, the definitions of the different classes of heavy-tailed distributions are provided.
 

 \begin{example}[Gamma scaling] \rm 
 		Consider $\Theta \sim \mbox{Gamma}(\alpha, 1)$, $\alpha>0$, with Laplace transform
 		\begin{align*}
 			\mathcal{L}_\Theta(u) = (1 + u)^{-\alpha}, \quad u\ge -1.
 		\end{align*}
	Then, the survival function $S_{Y}$ of $Y$ is given by
	\begin{align*}
		S_{Y}(y) =  \bfpi ( \mat{I} - h^{-1}(y)  \bfT)^{-\alpha} \bfe, \quad y\ge0. 
	\end{align*}
	As for the matrix-Pareto type II laws introduced in \cite{albrecher2021cph}, taking more general $\Theta \sim \mbox{Gamma}(\alpha, \gamma)$, $\gamma >0$, results in the same class of distributions. For this reason, we work only with $\mbox{Gamma}(\alpha, 1)$.
Consider now the particular case $\lambda(y) = \eta y^{\eta - 1}$, $\eta>0$, then
	\begin{align*}
		S_{Y}(y) =  \bfpi ( \mat{I} - y^{\eta}  \bfT)^{-\alpha} \bfe \,. 
	\end{align*}
	We call this the {\em Matrix-Burr} distribution.

Regarding the asymptotic behavior, we have that
\begin{align*}
	S_{Y}(y) \sim C (h^{-1}(y))^{-\alpha} \,,
\end{align*}
where $C$ is a positive constant, which follows from an eigenvalue decomposition of $\bfT$. The first-order precise asymptotics for the different intensities from Table~\ref{tab:trans} are provided in Table \ref{tab:gammaasym}, where $D$, $b$, and $c$ denote positive real-valued constants, which may change between intensities, but we write the same symbol for display purposes. Throughout the rest of this section, we use the same notational convention. 
\begin{table}[h]
\centering
\begin{tabular}{lccc}
  \hline
      Intensity              &   Precise asymptotics   &  Class  \\
  \hline
Pareto       & $D(\log(by))^{-\alpha}$ & Slowly varying  \\[2mm] 
Weibull      & $Dy^{-\alpha\eta}$    & Regularly varying \\ [2mm]
Lognormal    & $D(\log(y))^{-\alpha\eta}$    & Slowly varying \\ [2mm]
Loglogistic  & $D(\log(by))^{-\alpha}$     & Slowly varying    \\ [2mm]
Gompertz     & $D\exp( - b y )$     & Exponential      \\ 
   \hline
\end{tabular}
\caption{Asymptotics for Gamma scaling.}
\label{tab:gammaasym}
\end{table}
\end{example}

 \begin{example}[Positive stable scaling] \rm 
Consider  $\Theta$ positive stable with stability parameter $\alpha \in (0,1]$. Then
 	\begin{align*}
		S_{Y}(y) = \bfpi \exp( - (-\bfT)^{\alpha} (h^{-1}(y))^{\alpha}) \bfe, \quad y\ge0. 
	\end{align*}
As a particular case, take $\lambda(y) = \eta y^{\eta - 1}$, $\eta>0$. Then
\begin{align*}
		S_{Y}(y) = \bfpi \exp( - (-\bfT)^{\alpha} y^{\eta \alpha}) \bfe \,. 
\end{align*}
It was noted in \cite{albrecher2021cph} that $(\bfpi , - (-\bfT)^{\alpha})$ is a PH representation. Thus, some simple calculations show that these distributions span the same class as the matrix-Weibull laws introduced in \cite{albrecher2019inhomogeneous}. This is in contrast to the class of CPH distributions with stable mixing in \cite{albrecher2021cph}, which only span the matrix-Weibull laws with $\eta \in (0,1)$.  

Regarding their asymptotic behavior, we have
\begin{align*}
	S_{Y}(y) \sim C \exp(-b (h^{-1}(y))^{\alpha}) \,.
\end{align*}
Table \ref{tab:PSasym} gives the precise asymptotics for the different intensities of Table~\ref{tab:trans}.
\begin{table}[h]
\centering
\begin{tabular}{lccc}
  \hline
      Intensity              &   Precise asymptotics   &  Class  \\
  \hline
Pareto       & $ D \exp(-b (\log(cy))^{\alpha})$ &  Slowly varying \\[2mm]
Weibull      & $D \exp(-b y^{-\alpha\eta})$    & Weibull-type \\ [2mm]
Lognormal    & $D \exp(-b (\log(y))^{\alpha \eta})$    & \makecell{ Slowly varying for $\alpha \eta <1$ \\  Regularly varying for $\alpha\eta = 1$ \\  Lognormal-type for $\alpha\eta>1$ \\ [2mm]}  \\ 
Loglogistic  & $D \exp(-b (\log(cy))^{\alpha})$     & Slowly varying  \\[2mm] 
Gompertz     & $D \exp(-b \exp(c y))$     & Gumbel      \\ 
   \hline
\end{tabular}
\caption{Asymptotics for positive stable scaling.}
\label{tab:PSasym}
\end{table}
 \end{example}

\begin{example}[Inverse Gaussian scaling] \rm 
Consider inverse Gaussian scaling  with parameters $\nu >0$ and $\eta >0$ and density
\begin{align*}
	f_\Theta(\theta) = \frac{\sqrt{\eta}}{\sqrt{2 \pi \theta^3}} \exp \left( -\frac{\eta}{2 \nu^2 \theta} (\theta - \nu )^2 \right),\quad \theta>0.
\end{align*}
Then, the corresponding Laplace transform of $\Theta$ is given by
\begin{align*}
	\mathcal{L}_\Theta(u) = \exp \left( -\frac{\eta \sqrt{1 + 2 \nu ^2 u / \eta}}{\nu} + \frac{\eta}{\nu} \right),\quad u\ge0.
\end{align*}
We take the particular case $\nu = 1$ and $\sigma^2 = 1/\eta$. In this way 
\begin{align*}
	\mathcal{L}_\Theta(u) = \exp \left( \frac{1}{\sigma^2} \left(1 - \sqrt{1 + 2 \sigma ^2 u  }\right) \right) \,.
\end{align*}
Thus, 
\begin{align*}
	S_Y(y) =  \bfpi \exp \left( \frac{1}{\sigma^2} \left(\mat{I} - \sqrt{\mat{I} - 2 \sigma ^2  h^{-1}(y) \bfT  }\right) \right) \bfe, \quad y\ge0.
\end{align*}

Regarding the asymptotic behavior, we have that
\begin{align*}
	S_{Y}(y) \sim C \exp( -b (h^{-1}(y))^{1/2}) \,.
\end{align*}
Tables \ref{tab:IGasym} gives the precise asymptotics for the different intensities of Table~\ref{tab:trans}.
\begin{table}[h]
\centering
\begin{tabular}{lccc}
  \hline
      Intensity              &   Precise asymptotics   &  Class  \\
  \hline
Pareto       & $D\exp( -b (\log(cy))^{1/2})$ & Slowly varying  \\[2mm] 
Weibull      & $D\exp( -b y^{\eta/2})$    & Weibull-type  \\ [2mm]
Lognormal    & $D\exp( -b (\log(y))^{\eta/2})$    & \makecell{ Slowly varying for $ \eta <2$ \\  Regularly varying for $\eta = 2$ \\  Lognormal-type for $\eta>2$ \\ [2mm]} \\ [2mm]
Loglogistic  & $D\exp( -b (\log(cy))^{1/2})$     & Slowly varying    \\ [2mm]
Gompertz     & $D\exp( -b \exp( cy))$     & Gumbel      \\ 
   \hline
\end{tabular}
\caption{Asymptotics for inverse Gaussian scaling.}
\label{tab:IGasym}
\end{table}
\end{example}

\begin{example}[PVF scaling]\rm 
Consider the family of power variance function (PVF) distributions with Laplace transform 
\begin{align*}
	\mathcal{L}_{\Theta}(u) = \exp \left( \frac{\eta(1-\gamma)}{\gamma} \left( 1- \left( 1 + \frac{\nu u }{\eta (1- \gamma)} \right) ^{\gamma} \right) \right), \quad u\ge0,
\end{align*}
where $\nu >0 $, $\eta>0$ and $0 < \gamma \leq 1.$
This family includes the Gamma, inverse Gaussian and the positive stable distributions as particular cases. Here we assume that $\nu = 1$, which results in 
\begin{align*}
	S_Y(y) =  \bfpi \exp \left( \frac{\eta(1-\gamma)}{\gamma}\left(\mat{I} - \left({\mat{I} -    \frac{h^{-1}(y)}{\eta  (1 - \gamma)} \bfT  }\right)^{\gamma}\right) \right) \bfe,\quad y\ge0.
\end{align*}

Regarding the asymptotic behavior, we have that
\begin{align*}
	S_{Y}(y) \sim C \exp(-b(h^{-1}(y))^{\gamma}) \,,
\end{align*}
which results in the same asymptotics of Table~\ref{tab:PSasym} for the positive stable case, but with $\alpha$ replaced by $\gamma$.

\end{example}

 \begin{example}[Compound Poisson scaling] \rm 
 Consider a compound model $\Theta = \sum_{i = 1}^N V_i$ with $V_1, V_2, \dots $ i.i.d.\ random variables independent of $N$. 
	In general, the Laplace transform of $\Theta$ is given by 
	\begin{align*}
		\mathcal{L}_\Theta(u) = \mathcal{L}_N (- \log \mathcal{L}_V(u)),\quad u\ge0.
	\end{align*}
In particular, for $V\sim \mbox{Gamma}(\alpha, 1)$ and $N \sim \mbox{Poisson}(\rho)$, we obtain  
\begin{align*}
	\mathcal{L}_\Theta(u) = \exp \left( - \rho \left( 1- \left({1 + u} \right)^{-\alpha} \right) \right) \,.
\end{align*}
Thus,
\begin{align*}
	S_Y(y) = \bfpi \exp \left( - \rho \left( \mat{I}-  \left( \mat{I} - h^{-1}(y) \bfT \right)^{-\alpha} \right) \right) \bfe, \quad y\ge0.
\end{align*}

Note that this distribution has an atom at infinity with probability $\exp(-\rho)$, corresponding to the probability of $\P(N= 0)$. In survival analysis terms, this means that an individual may never experience the event of interest with such probability. Considering $N+1$ instead of $N$ removes such an atom.

\end{example}


\begin{example}[Discrete scaling] \rm 
Assume that $\Theta$ is a discrete random variable taking values in $ \{\eta_1, \eta_2, \dots \}\subset \mathbb{R}_+$ with corresponding probabilities $\bfalp = (\alpha_1, \alpha_2, \dots)$, that is, $\P(\Theta= \eta_i ) = \alpha_i$, $i=1,2,\dots$. Then, 
\begin{align*}
	S_Y(y) = \sum_i \alpha_i \bfpi \exp\left(\eta_i \bfT h^{-1}(y) \right) \bfe, \quad y\ge0.
\end{align*}
Define the linear transformation $\tilde{\bfT}$ on $\mathbb{R}^\mathbb{N}$ given by
\begin{align*}
	\tilde{\bfT} = \left( \begin{matrix}
  \bfT \eta_1 & \mat{0}& \cdots \\
  \mat{0} &\bfT \eta_2 & \cdots\\
  \vdots & \vdots & \ddots
\end{matrix} \right)\,.
\end{align*}
Then, we can rewrite the survival function of $Y$ as 
\begin{align*}
	S_Y(y) =  (\bfalp \otimes \bfpi) \exp\left( \tilde{\bfT} h^{-1}(y)  \right) \tilde{\bfe}, \quad y\ge0,
\end{align*}
where $\otimes $ denotes the Kronecker product, and $\tilde{\bfe} $ is a column vector of ones of appropriate dimension. This can be thought of as an infinite-dimensional IPH distribution. The case $\lambda \equiv 1$ recovers the class of NPH distributions introduced in \cite{bladt2015calculation}.

Note that another approach to study the asymptotic behavior, and that is particularly convenient in the discrete scaling case, is to use the representation $Y = h(Z/\Theta)$, so that 
\begin{align*}
	\P(Y >y ) = \P( Z/\Theta > h^{-1}(y)) = S_{Z/\Theta }(h^{-1}(y)) \,,
\end{align*}
and employ the asymptotics of $Z/\Theta$. For instance, taking $\Theta \sim \mbox{Gamma}(\alpha, 1)$, we have that $Z/\Theta$ is regularly varying with index $\alpha$ (see \cite{albrecher2021cph} for details). This leads to the same asymptotic results in Table~\ref{tab:gammaasym} for the different choices of intensities $\lambda$. For the discrete scaling, we could take, for instance, $\Theta$ with Zeta distribution leading to the same asymptotic results. 


As a second case, take $V := 1/\Theta$ with Weibull-type tail so that $V Z$ has Weibull-type tail with shape parameter in $(0,1)$ (see \cite{rojas2018asymptotic}). Thus, the asymptotic behavior for the different intensities resemble those in Table~\ref{tab:PSasym}.
\end{example}

\begin{example}[Missing covariates in the proportional intensities model] \rm 
Consider the proportional intensities model (also known as PH regression) introduced in \cite{bladt2020survival} with vectors of observed and unobserved covariates $\bfX_1$ and $\bfX_2$, respectively. Namely, the intensity is of the form
\begin{align*}
	\lambda(t; \bfX_1, \bfX_2) = \lambda(t) \exp ( \bfbeta_1 \bfX_1  +  \bfbeta_2 \bfX_2 ),\quad t\ge0.
\end{align*}
 Given that the vector $\bfX_2$ is unknown, the model cannot be employed in practice. However, we can assume that 
\begin{align*}
	\Theta := \exp( \bfbeta_2 \bfX_2 )
\end{align*}
is an unobserved random variable independent of $\bfX_1$. In this way, the scaled intensity model can be employed to account for the effect of omitted covariates by considering a parametric model for $\Theta$. Such additional random component can thus help account for additional variability observed in data that cannot be explained by a simpler model.
\end{example}

 \subsection{Parameter estimation} \label{sec:siphest}
 In order to derive an EM algorithm for SIPH distribution, we first recall the corresponding algorithm for CPH distributions in \cite{albrecher2021cph} (see \cite{bladt2017fitting} for the discrete scaling case).
 Consider $y_1, \dots, y_K$ an i.i.d.\ sample from a CPH distributed random variable $Y$, which we will also denote by $\bfy$. Here, we assume that the scaling component $\Theta$ belongs to a parametric family depending on the parameter vector $\bfalp$ and denote by $f_\Theta$ its corresponding density.  We now make the following definitions. Let $B_k$ be the number of times the underlying Markov jump process of $Y$ starts
in state $k$, $N_{kl}$ the total number of transitions from state $k$ to $l$ until absorption, $N_k$ the number of
times that $k$ was the last state to be visited before absorption, and finally, let $Z_k$ be  the  cumulated time that
the Markov jump process spent in state $k$. The detailed routine for estimation of CPH distributions is given in Algorithm~\ref{alg:EMCPH}.
 
\begin{algorithm}[]
\caption{{EM algorithm for CPH distributions}}\label{alg:EMCPH}
 \begin{algorithmic}
 	\State \textit{\textbf{Input}}: Initialize with $( \bfpi,\bfT, \vect{\alpha})$.
	\begin{itemize}
	\item[ 1.] \textit{E-step:} Calculate

\begin{align*}
	& \E \left( B_{k} \mid \bfY = \bfy \right) 
	=\sum_{n=1}^{K} \int_0^\infty \frac{  \pi_k \, \bfe^{\top}_{k} \exp({\theta \bfT y_n }) \theta \bft }{f_{Y}(y_n)} \, f_\Theta(\theta) d\theta   \\[3mm]
	&\E \left( \Theta Z_{k} \mid \bfY = \bfy \right)
	 = \sum_{n=1}^{K} \int_0^\infty  \theta \frac{ \int^{y_n}_{0}  \bfe^{\top}_{k} \exp({\theta \bfT(y_n-u)})\theta \bft \bfpi \exp({\theta \bfT u })\bfe_{k}   du }{f_{Y}(y_n)}  f_\Theta(\theta) d\theta  \\[3mm]
	& \E \left( N_{kl} \mid \bfY = \bfy \right)
	 = \sum_{n=1}^{K} \int_0^\infty \theta t_{kl} \frac{ \int^{y_n}_{0}  \bfe^{\top}_{l} \exp({\theta \bfT(y_n-u)})\theta \bft \bfpi \exp({\theta \bfT u })\bfe_{k}   du }{f_{Y}(y_n)}  f_\Theta(\theta) d\theta \\[3mm]
	& \E \left( N_{k} \mid \bfY = \bfy \right)
	 = \sum_{n=1}^{K} \int_0^\infty \theta t_{k}   \frac{\bfpi \exp({\theta \bfT y_n })\bfe_{k}    }{f_{Y}(y_n)} f_\Theta(\theta) d\theta
\end{align*}

	\item[ 2.]  \textit{M-step:}  Let
	\begin{align*}
	\hat{\vect{\alpha}} &= \argmax_{\vect{\alpha}} \E \left( \log(f_\Theta(\Theta; \vect{\alpha}) ) \mid \bfY = \bfy \right) \\
	 &= \argmax_{\vect{\alpha}} \sum_{n=1}^{K} \int_0^\infty \log(f_\Theta(\theta; \vect{\alpha}) ) \,  \frac{ \bfpi \exp({\theta \bfT y_n}) \theta \bft \, }{f_{Y}(y_n)}f_\Theta(\theta) d\theta
	\end{align*}
	\begin{align*}
		&\hat{\pi}_{k} = \frac{\E\left( B_{k} \mid \bfY = \bfy \right)}{K}  
		\,, \quad
		\hat{t}_{kl} = \frac {\mathlarger{ \E\left( N_{kl} \mid \bfY = \bfy \right) }}{\mathlarger{ \E \left(\Theta Z_{k} \mid \bfY = \bfy \right) }}
		\,, \quad
		\hat{t}_{k} = \frac {\mathlarger{ \E\left( N_{k} \mid \bfY = \bfy \right) }} {\mathlarger{ \E\left( \Theta Z_{k} \mid \bfY = \bfy \right)}}
		\,, \\
		&\hat{t}_{kk} = -\sum_{l \neq k} \hat{t}_{kl} -\hat{t}_{k} \,.
	\end{align*}
	Let $\hat{\bfpi} = ( \hat{\pi}_{1}, \ldots , \hat{\pi}_{p} )$, $\hat{\bfT} = \{ \hat{t}_{kl} \}_{ k, l = 1, \ldots, p}$, and $\hat{\bft} = ( \hat{t}_{1}, \ldots, \hat{t}_{p} )^{ \top }$.
	
	\item[ 3.]  Assign $\vect{\alpha} = \hat{\vect{\alpha}} $, $\bfpi:=\hat{\bfpi}$, $\bfT :=\hat{\bfT}$, $ \bft :=\hat{\bft}$ and GOTO 1.
	 \end{itemize}
 	 \State \textit{\textbf{Output}}: Fitted parameters $( \bfpi,\bfT , \vect{\alpha} )$.
	 \end{algorithmic}
\end{algorithm}
 
We now derive a generalized EM algorithm for maximum-likelihood estimation of SIPH distributions. 
Assume that $\lambda(\,\cdot\, ; \vect{\eta})\ge0$ is a nonnegative parametric function depending on the vector $\vect{\eta}$. 
Let $Y \sim \mbox{SIPH}(\bfpi, \bfT, \lambda(\,\cdot\, ; \vect{\eta}), \Theta, \bfbeta)$, then $$Y \eqd h(\exp(-\bfbeta \bfX ) Z /\Theta; \vect{\eta}),$$ where $Z \sim \mbox{PH}(\bfpi, \bfT)$. In particular, this implies that $h^{-1}(Y; \vect{\eta}) \exp(\bfbeta \bfX ) \eqd Z /\Theta$, meaning that $h^{-1}(Y; \vect{\eta}) \exp(\bfbeta \bfX )$ is scaled PH distributed. Consider now $y_1,\dots,y_K$ an i.i.d.\ sample from this $Y$, then the EM algorithm for parameter estimation is the following.

 \begin{algorithm}[]
 \caption{{Generalized EM algorithm for SIPH distributions}}\label{alg:SIPH} 
 \begin{algorithmic}
 	\State \textit{\textbf{Input}}: Initialize with $( \bfpi,\bfT , \vect{\alpha}, \vect{\eta}, \bfbeta )$.
 	\begin{itemize}
 	\item[ 1.] Transform the data into $z_n=h^{-1}(y_n; \vect{\eta}) \exp(\bfbeta\bfX_n ) $, $n=1,\dots,K$, and 
 	apply the E- and M-steps of Algorithm~\ref{alg:EMCPH} by which we obtain the estimators $( \hat{\bfpi},\hat{\bfT}, \hat{\vect{\alpha}})$. 
 	
 	\item[ 2.] Compute  
 	\begin{align*}
 		(\hat{\vect{\eta}}, \hat{\bfbeta})  & = \argmax_{(\vect{\eta}, \bfbeta)} \sum_{n=1}^{K} \log (f_{Y}(y_n; \hat{\bfpi}, \hat{\bfT}, \hat{\vect{\alpha}}, \vect{\eta}, \bfbeta )) \,. 
 	\end{align*}
 	
 	\item[ 3.] Assign $(\vect{\pi},\mat{T},\vect{\alpha},\vect{\eta}, \bfbeta) =(\hat{\vect{\pi}},\hat{\mat{T}}, \hat{\vect{\alpha}}, \hat{\vect{\eta}}, \hat{\bfbeta})$ and GOTO 1. 
 	\end{itemize}
 	 \State \textit{\textbf{Output}}: Fitted parameters $( \bfpi,\bfT , \vect{\alpha}, \vect{\eta}, \bfbeta )$.
 \end{algorithmic}
 \end{algorithm}
 \begin{proposition}\label{prop:emconv}
 	Algorithm~\ref{alg:SIPH} increases the likelihood function at each iteration. Since for fixed $p$, the likelihood of SIPH distributions is bounded, {\color{black} convergence towards a (possibly local) maximum is guaranteed.}
 \end{proposition}

 \begin{proof}
 	By the change of variable theorem, we have that 
 	\begin{align*}
 		f_{Y}(y) 
 		& = f_{Z/\Theta}(h^{-1}(y;\vect{\eta})\exp(\bfbeta \bfX ) ; \vect{\pi},\mat{T},\vect{\alpha}) \lambda (y;\vect{\eta}) \exp(\bfbeta \bfX ) ,\quad y\ge0.
 	\end{align*}
 	Consider parameter values $(\vect{\pi}_i,\mat{T}_i,\vect{\alpha}_i,\vect{\eta}_i, \bfbeta_i) $ after the $i$-th iteration. Then the data log-likelihood after the $i$-th iteration is given by 
 		\begin{eqnarray*} 
 		\lefteqn{ l( \vect{\pi}_i,\mat{T}_i,\vect{\alpha}_i,\vect{\eta}_i, \bfbeta_i ; \bfy, \bfX)} \\
 		& & = \sum_{n = 1}^{K} \log( f_{Z/\Theta}(h^{-1}(y_n;\vect{\eta}_i) \exp(\bfbeta_i \bfX_n ) ;\vect{\pi}_i,\mat{T}_i,\vect{\alpha}_i )) + \log( \lambda (y_n;\vect{\eta}_i) ) + \bfbeta_i\bfX_n  \,.
 	\end{eqnarray*}
 	
 	In the $(i + 1)$-th iteration, we first obtain $(\vect{\pi}_{i+1},\mat{T}_{i+1},\vect{\alpha}_{i+1})$ in 1. so that
 	\begin{eqnarray*} 
 		\lefteqn{l( \vect{\pi}_i,\mat{T}_i,\vect{\alpha}_i,\vect{\eta}_i, \bfbeta_i ; \bfy, \bfX)} \\
 		& & \leq \sum_{n = 1}^{K} \log( f_{Z/\Theta}(h^{-1}(y_n;\vect{\eta}_i) \exp(\bfbeta_i \bfX_n ) ;\vect{\pi}_{i+1},\mat{T}_{i+1},\vect{\alpha}_{i+1} )) + \log( \lambda (y_n;\vect{\eta}_i))+ \bfbeta_i \bfX_n \\
 		& & = l( \vect{\pi}_{i+1},\mat{T}_{i+1},\vect{\alpha}_{i+1},\vect{\eta}_{i}, \bfbeta_{i} ; \bfy, \bfX) \,.
 	\end{eqnarray*}
 	Finally, by 2.
 	 	\begin{align*}
 		l( \vect{\pi}_i,\mat{T}_i,\vect{\alpha}_i,\vect{\eta}_i, \bfbeta_i ; \bfy, \bfX)
 		& \leq  \argmax_{(\vect{\eta}, \bfbeta)}  l( \vect{\pi}_{i+1},\mat{T}_{i+1},\vect{\alpha}_{i+1},\vect{\eta}, \bfbeta ; \bfy, \bfX) \\
 		& =  l( \vect{\pi}_{i+1},\mat{T}_{i+1},\vect{\alpha}_{i+1},\vect{\eta}_{i+1}, \bfbeta_{i+1} ; \bfy, \bfX) \,.
 	\end{align*}
 \end{proof}
 
 \begin{remark}\rm
 The optimization problem
 \begin{align}\label{secon_optim_eq_SIPH}
 \argmax_{(\vect{\eta}, \bfbeta)} \sum_{n=1}^{K} \log (f_{Y}(y_n; \hat{\bfpi}, \hat{\bfT}, \hat{\vect{\alpha}}, \vect{\eta}, \bfbeta ))
 \end{align}
 of Algorithm \ref{alg:SIPH} is computationally heavy. However, observe that fewer iterations of any optimization routine {\color{black}are} sufficient for the proof and conclusion of Proposition \ref{prop:emconv} to hold, and full convergence of \eqref{secon_optim_eq_SIPH} is not necessary. For instance, one step of the $\argmax$ routine can already provide good results. \end{remark}
 
 \begin{remark}[Incorporating right-censoring]\rm

Algorithm~\ref{alg:SIPH} can be modified to work with censored data. We illustrate the changes by considering only the case of right-censoring since it is the most common scenario in survival analysis applications.
 However, left-censoring and interval-censoring can be treated by similar means. 
 In such a case, we no longer observe $Y = y$ but instead only that $Y \in [v, \infty)$.
 By monotonicity of $h$, we have that $h^{-1}(Y; \vect{\eta}) \exp(\bfbeta \bfX ) \in [h^{-1}(v; \vect{\eta}) \exp(\bfbeta \bfX ) , \infty )$, which can be interpreted as a censored observation of a scaled PH distributed random variable. Moreover, in \cite{albrecher2021cph} (and \cite{bladt2017fitting}), a modified EM algorithm for the estimation of scaled PH distributions is presented for the case of censored observations. This means that the main change in Algorithm~\ref{alg:SIPH} is in step 2, where we must now compute 
  	\begin{align*}
 		(\hat{\vect{\eta}}, \hat{\bfbeta})  
 		 = \argmax_{(\vect{\eta}, \bfbeta)} &\sum_{n \,:\, y_n \,\text{observed}}^{K} \log (f_{Y}(y_n; \hat{\bfpi}, \hat{\bfT}, \hat{\vect{\alpha}}, \vect{\eta}, \bfbeta )) \\
 		&  + \sum_{n \,:\, y_n \,\text{censored}}^{K} \log (S_{Y}(y_n; \hat{\bfpi}, \hat{\bfT}, \hat{\vect{\alpha}}, \vect{\eta}, \bfbeta )) \,.
 	\end{align*}
\end{remark}

\subsection{Estimation for fractional PH distributions}
 
 	A key distinction of the matrix Mittag-Leffler distribution (or fractional PH), with respect to the other models introduced in Section~\ref{sec:example}, is that the transformation $h(x) =  x^{1/\alpha}$ and the mixing distribution $\Theta = 1/S_{\alpha}^{\alpha}$ depend on the same parameter $\alpha$. This makes statistical estimation very challenging by ad-hoc methods, and thus embedding into the SIPH class is useful for this purpose. Note that the transformation parameters are different from the scaling component's parameters for the previously presented models, and this last scenario is the central assumption in the derivation of Algorithm ~\ref{alg:SIPH}. Thus, special treatment must be taken for the estimation of matrix Mittag-Leffler distributions when seen as SIPH distributions. This is now solved by employing a modified EM algorithm, the details given in Algorithm \ref{alg:MML}.
 \begin{algorithm}[]
 \caption{{EM algorithm for matrix Mittag-Leffler distributions}}\label{alg:MML} 
 \begin{algorithmic}
 	\State \textit{\textbf{Input}}:  Initialize with $( \bfpi,\bfT , \alpha )$.
 	\begin{enumerate}
 		
	\item Transform the data into $z_n=(y_n)^{\alpha}$, $n=1,\dots,K$, and 
 	apply the E- and M-steps of Algorithm~\ref{alg:EMCPH} with fixed $\alpha$ to obtain the estimators $( \hat{\bfpi},\hat{\bfT})$. 
 	\item Compute  
 	\begin{align*}
 		\hat{{\alpha}}  & = \argmax_{\alpha} \sum_{n=1}^{K} \log (f_{Y}(y_n; \hat{\bfpi}, \hat{\bfT}, \alpha )) \,.
 	\end{align*}
 	\item Assign $(\vect{\pi},\mat{T},{\alpha}) =(\hat{\vect{\pi}},\hat{\mat{T}}, \hat{{\alpha}})$ and GOTO 1.  
 	\end{enumerate}
 	 \State \textit{\textbf{Output}}: Fitted parameters $( \bfpi,\bfT , \alpha )$.
 \end{algorithmic}
 \end{algorithm}

By the same method of proof of Algorithm~\ref{alg:SIPH}, one can show that Algorithm \ref{alg:MML} increases the likelihood in each iteration, and hence we omit the details for brevity.

\section{Shared scaling}\label{sec:mult}
This section presents a multivariate extension of SIPH distributions, inspired by the construction principle of the shared frailty model. The key idea is to think of an underlying random variable which is a common scaling factor to all the coordinates of an independent random vector, creating dependency and heavy-tailedness all at once through the same mechanism.

\subsection{A class of multivariate CPH distributions}\label{sec:MCPH}
Before going into full generality, we consider the case where there is no deterministic time-transform component. This allows for a more transparent treatment with explicit formulas. Thus, consider the conditionally independent random variables $\bfY =(Y_1,\dots,Y_d)^{\top}$ given $\Theta=\theta$ such that 
\[  Y_i \mid \Theta = \theta \sim \mbox{PH}(\vect{\pi}_i,\theta \mat{T}_i) \,, \quad i = 1,\dots,d \,. \]
Then, the joint survival function of $\bfY $ is given by 
\begin{align*}
	S_{\bfY}(\bfy  )
	& = \int \P (Y_1>y_1,\dots,Y_d>y_d \mid  \Theta=\theta)dF_\Theta (\theta) \\
	& = \int \prod_{i=1}^d \vect{\pi}_i \exp\left({\theta \mat{T}_i y_i}\right)\vect{e}dF_\Theta (\theta) \\
	& = \int (\vect{\pi}_1\otimes\cdots \otimes \vect{\pi}_d )\exp\left({\theta (\mat{T}_1y_1\oplus\cdots\oplus\mat{T}_dy_d)}\right)\vect{e}\ dF_\Theta (\theta) \\
	& =   (\vect{\pi}_1\otimes\cdots \otimes \vect{\pi}_d ) \mathcal{L}_\Theta(-(\mat{T}_1y_1\oplus\cdots\oplus\mat{T}_d y_d)) \bfe, \quad y_i\ge0 \:\: i=1,\dots,d,
\end{align*}
where $\oplus, \otimes$ denote the Kronecker sum and product, respectively. 
In particular, this yields the joint density 
\begin{align*}
	f_{\bfY}(\bfy  )
	& =  (-1)^{d}  (\vect{\pi}_1\otimes\cdots \otimes \vect{\pi}_d ) \mathcal{L}_\Theta^{(d)}(-(\mat{T}_1y_1\oplus\cdots\oplus\mat{T}_d y_d)) \tilde{\bft} ,\quad y_i\ge0 \:\: i=1,\dots,d,
\end{align*}
where $ \tilde{\bft} = \vect{t}_1\otimes\cdots \otimes \vect{t}_d$ and $ \mathcal{L}_\Theta^{(d)}(u)$ is the derivative of order $d$ of $\mathcal{L}_\Theta(u)$, which can again be shown by the use of functional calculus through Cauchy's formula.
Moreover, marginally we get continuously scaled PH behavior: 
\[  Y_i \sim \mbox{CPH}(\vect{\pi}_i,\mat{T}_i,\Theta) \,, \quad i = 1, \dots, d \,.  \]

Alternatively, it is easy to see that $\bfY$ has representation $(Y_1, \dots, Y_d)^{\top} = (Z_1, \dots ,Z_d)^{\top}/  \Theta $, where  $Z_i$ are independent $\mbox{PH}(\vect{\pi}_i,\mat{T}_i)$ distributed random variables independent of $\Theta$, $i =1, \dots, d$. Indeed, 
\begin{align*}
	\P(Y_1 > y_1, \dots ,Y_d >y_d ) &= \int \P(Y_1 > y_1, \dots ,Y_d >y_d \mid \Theta = \theta ) dF_\Theta(\theta) \\
	& = \int \P(Z_1 > \theta y_1, \dots ,Z_d>\theta y_d \mid \Theta = \theta ) dF_\Theta(\theta) \\
	& = \int \prod_{i=1}^d \vect{\pi}_i \exp\left({\theta \mat{T}_i y_i}\right)\vect{e} dF_\Theta(\theta) \\
	& =  S_{\bfY}(\bfy  ) \,.
\end{align*}

These multivariate distributions were studied from another perspective in \cite{furman2021multiplicative}, where the authors derived some properties in the context of risk management. We presently derive some probabilistic properties, provide an estimation method, and extend the class to allow for deterministic time transforms. In the next section we also allow for scaling of different components of the random vector by different (but correlated) scaling random variables. Since these distributions will be the building blocks of the more general time-inhomogeneous multivariate models presented in Section~\ref{subsec:MSIPH}, a good understanding of the former facilitates the {\color{black}treatment} of the latter.


  \begin{example}[Gamma scaling] \rm 
Consider  $\Theta\sim \mbox{Gamma}(\alpha, 1)$, $\alpha>0$, then the joint survival function of $\bfY$ is given by
\begin{align*}
	S_{\bfY }(\bfy ) = (\vect{\pi}_1\otimes\cdots \otimes \vect{\pi}_d ) \left(\mat{I}-(\mat{T}_1 y_1\oplus\cdots\oplus\mat{T}_d y_d)\right)^{-\alpha} \bfe, \quad y_i\ge0 \:\: i=1,\dots,d.
\end{align*}
This distribution can be seen to be a matrix version of Mardia's multivariate Pareto distribution (see \cite{mardia1962multivariate}).  
 
\end{example}

\subsection{Parameter estimation: multivariate CPH distributions}\label{subsec:MCPHest}

We now present a generalized EM algorithm for maximum-likelihood estimation of the class of multivariate CPH distributions introduced previously. 
The complete data is the scaling component $\Theta$ together with the conditionally independent Markov jump processes paths. We further assume that $\Theta$ belongs to a parametric family depending on the vector $\bfalp$ and denote by $f_\Theta$ its corresponding density.

Consider observations $\bfy_n = (y_{n}^{(1)}, \dots, y_{n}^{(d)})^{\top}$, $n =1 ,\dots, K$, from a multivariate CPH distributed random vector, and let $\tilde{\bfy}$ denote the whole data set. 
We also denote by $\tilde{\bfpi}$ and $\tilde{\bfT}$ the sets of parameters $\{\bfpi_1, \dots, \bfpi_d\}$ and $\{\bfT_1,\dots, \bfT_d\}$, respectively, and $ \pi_{k}^{(i)}$ and $t_{kl}^{(i)}$ to refer to the entries of $\bfpi_i$ and $\bfT_i$, $i = 1, \dots, d$. In order to write down the complete likelihood $L_c(\tilde{\bfpi}, \tilde{\bfT},\vect{\alpha} ;\tilde{\bfy})$, we need the following definitions. 
For each $i = 1, \dots ,d$, 
let $B_k^i$ be the number of times the underlying Markov jump process of $Y_i$ starts
in state $k$, $N_{kl}^i$ the total number of transitions from state $k$ to $l$ until absorption, $N_k^i$ the number of
times that $k$ was the last state to be visited before absorption, and finally, let $Z_k^i$ be  the  cumulated time that
the Markov jump process spent in state $k$.

Then, the complete likelihood is given by
\begin{align*}
&L_c(\tilde{\bfpi}, \tilde{\bfT}, \vect{\alpha} ; \tilde{\bfy} ) \\ 
&\quad = f_\Theta(\theta; \vect{\alpha}) \prod_{i=1}^{d} \prod_{k=1}^{p_i}(\pi_{k}^{(i)})^{B_{k}^i}\prod_{k=1}^{p_i}\prod_{l=1, l\neq k}^{p_i} \left(\theta t_{kl}^{(i)}\right)^{N_{kl}^i}\exp\big( -\theta t_{kl}^{(i)}Z_{k}^i \big) \\
&\quad\quad\times\prod_{k=1}^{p_i}\left(\theta t_{k}^{(i)} \right)^{N_{k}^i}\exp\big(-\theta t_{k}^{(i)} Z_{k}^i \big)  \,,
\end{align*}
with corresponding log-likelihood (discarding the terms which do not depend on any parameters)
\begin{eqnarray*}
\lefteqn{ l_c(\tilde{\bfpi}, \tilde{\bfT},\vect{\alpha} ; \tilde{\bfy}) }\\
& &= \sum_{i=1}^{d} \sum_{k=1}^{p_i} {B_{k}^i} \log \left( \pi_{k}^{(i)} \right) +  \sum_{i=1}^{d} \sum_{k=1}^{p_i}\sum_{l=1, l\neq k}^{p_i} N_{kl}^i \log \left(t_{kl}^{(i)}\right)
- \sum_{i=1}^{d} \sum_{k=1}^{p_i}\sum_{l=1, l\neq k}^{p_i}{t_{kl}^{(i)} \theta Z_{k}^{i} } \\
& & \quad +   \sum_{i=1}^{d} \sum_{k=1}^{p_i} {N_{k}^i}\log \left(t_{k}^{(i)}\right)
- \sum_{i=1}^{d} \sum_{k=1}^{p_i} {t_{k}^{(i)} \theta Z_{k}^i } + \log(f_\Theta(\theta; \vect{\alpha}) )  \,.
\end{eqnarray*}

Regarding the E-step, which consists of computing the conditional expectation of the log-likelihood given the observed data, the calculations are somewhat similar to those of \cite{albrecher2021cph}. We illustrate the procedure by computing the conditional expectation of the logarithmic term. 
Consider one (generic) data point ($K = 1$) and let $\bfy = \bfy_1$. Then 
\begin{align*}
	\E \left[ \log(f_\Theta(\Theta; \vect{\alpha}) ) \mid \bfY=\bfy \right] & = \int_0^\infty \log(f_\Theta(\theta; \vect{\alpha}) ) f_{\Theta | \bfY} (\theta | \bfy) d\theta \\
	& =   \int_0^\infty \log(f_\Theta(\theta; \vect{\alpha}) ) \frac{f_{\Theta , \bfY} (\theta , \bfy)}{f_{\bfY}(\bfy)} d\theta \\
	& =   \int_0^\infty \log(f_\Theta(\theta; \vect{\alpha}) ) \frac{f_{ \bfY | \Theta } ( \bfy | \theta ) f_\Theta(\theta)}{f_{\bfY}(\bfy)} d\theta \\
	& = \int_0^\infty \log(f_\Theta(\theta; \vect{\alpha}) ) \frac{  \prod_{i =1 }^{d} \bfpi_i \exp({ \theta \bfT_i y^{(i)}}) \theta \bft_i }{f_{\bfY}(\bfy)} f_\Theta(\theta) d\theta \,.
\end{align*}
The formulas for all the other statistics are derived by similar calculations. 

Concerning the M-step, consisting of maximizing the conditional expected log-likelihood in terms of the parameters, for the parameter $\bfalp$ of the scaling component we have in full generality 
\begin{align*}
	\hat{\vect{\alpha}} &= \argmax_{\vect{\alpha}} \E \left( \log(f_\Theta(\Theta; \vect{\alpha}) ) \mid \tilde{\bfY} = \tilde{\bfy}  \right) \,.
\end{align*}
Regarding the PH component's parameters, the entries of the sub-intensity matrix can be found by direct differentiation of the log-likelihood, while for the vector of initial probabilities, we can employ a Lagrange multiplier argument. We omit further details for brevity. 
We summarize the complete procedure in Algorithm \ref{alg:EMMCPH}.

\begin{algorithm}[]
\caption{Generalized EM algorithm for multivariate CPH distributions}\label{alg:EMMCPH}
\begin{algorithmic}
	\State \textit{\textbf{Input}}: Initialize with $( \bfpi_1,\dots,\bfpi_d, \bfT_1,\dots, \bfT_d, \vect{\alpha})$.
	\begin{enumerate} 
	\item[ 1.]\textit{E-step:}  For each $i =1 ,\dots, d$, calculate

\begin{align*}
	& \E \left( B_{k}^i \mid \tilde{\bfY} = \tilde{\bfy} \right) 
	=\sum_{n=1}^{K} \int_0^\infty \frac{  \pi_{k}^{(i)} \, \bfe^{\top}_{k} \exp({\theta \bfT_i y_{n}^{(i)} }) \theta \bft_i \prod_{j \neq i }^{d} \bfpi_j \exp({ \theta \bfT_j y_{n}^{(j)}}) \theta \bft_j }{f_{\bfY}(\bfy_n)} \, f_\Theta(\theta) d\theta   \\[3mm]
	&\E \left( \Theta Z_{k}^i \mid \tilde{\bfY} = \tilde{\bfy} \right)\\
	 &  = \sum_{n=1}^{K} \int_0^\infty  \theta \frac{ \int^{y_n^{(i)}}_{0}  \bfe^{\top}_{k} \exp({\theta \bfT_i(y_n^{(i)}-u)})\theta \bft_i \bfpi_i \exp({\theta \bfT_i u })\bfe_{k}   du \prod_{j \neq i }^{d} \bfpi_j \exp({ \theta \bfT_j y_n^{(j)}}) \theta \bft_j }{f_{\bfY}(\bfy_n)}  f_\Theta(\theta) d\theta  \\[3mm]
	& \E \left( N_{kl}^i \mid \tilde{\bfY} = \tilde{\bfy}\right) \\
	 &  =  \sum_{n=1}^{K} \int_0^\infty \theta t_{kl}^{(i)} \frac{ \int^{y_n^{(i)}}_{0}  \bfe^{\top}_{l} \exp({\theta \bfT_i(y_n^{(i)}-u)})\theta \bft_i \bfpi_i \exp({\theta \bfT_i u })\bfe_{k}   du \prod_{j \neq i }^{d} \bfpi_j  \exp({ \theta \bfT_j y_n^{(j)}}) \theta \bft_j }{f_{\bfY}(\bfy_n)}  f_\Theta(\theta) d\theta \\[3mm]
	& \E \left( N_{k}^i \mid \tilde{\bfY} = \tilde{\bfy} \right)
	 = \sum_{n=1}^{K} \int_0^\infty \theta t_{k}^{(i)}   \frac{\bfpi_i \exp({\theta \bfT_i y_n^{(i)} })\bfe_{k}   \prod_{j \neq i }^{d} \bfpi_j \exp({ \theta \bfT_j y_n^{(j)}}) \theta\, \bft_j  }{f_{\bfY}(\bfy_n)} f_\Theta(\theta) d\theta
\end{align*}

	\item[2.]\textit{M-step:} Let
	\begin{align*}
	\hat{\vect{\alpha}} &= \argmax_{\vect{\alpha}} \E \left( \log(f_\Theta(\Theta; \vect{\alpha}) ) \mid \tilde{\bfY} = \tilde{\bfy}  \right) \\
	 &= \argmax_{\vect{\alpha}} \sum_{n=1}^{K} \int_0^\infty \log(f_\Theta(\theta; \vect{\alpha}) )   \frac{ \prod_{ i = 1}^{d} \bfpi_i \exp({ \theta \bfT_i y_n^{(i)}}) \theta \bft_i }{f_{\bfY}(\bfy_n)}f_\Theta(\theta) d\theta
	\end{align*}
	and 
	\begin{align*}
		&\hat{\pi}_{k}^{(i)} = \frac{\E\left( B_{k}^i \mid \tilde{\bfY} = \tilde{\bfy}  \right)}{K}  
		\,, \quad
		\hat{t}_{kl}^{(i)} = \frac {\mathlarger{ \E\left( N_{kl}^i \mid \tilde{\bfY} = \tilde{\bfy}  \right) }}{\mathlarger{ \E \left(\Theta Z_{k}^i \mid \tilde{\bfY} = \tilde{\bfy}  \right) }}
		\,, \quad
		\hat{t}_{k}^{(i)} = \frac {\mathlarger{ \E\left( N_{k}^i \mid \tilde{\bfY} = \tilde{\bfy}  \right) }} {\mathlarger{ \E\left( \Theta Z_{k}^i \mid \tilde{\bfY} = \tilde{\bfy}  \right)}}
		\,, \\
		&\hat{t}_{kk}^{(i)} = -\sum_{l \neq k} \hat{t}_{kl}^{(i)} -\hat{t}_{k}^{(i)} \,, \quad i = 1, \dots, d\,.
	\end{align*}
	Let $\hat{\bfpi}_i = ( \hat{\pi}_{1}^{(i)}, \ldots , \hat{\pi}_{p_i}^{(i)} )$, $\hat{\bfT}_i = \{ \hat{t}_{kl}^{(i)} \}_{ k, l = 1, \ldots, p_i}$, and $\hat{\bft}_i = ( \hat{t}_{1}^{(i)}, \ldots, \hat{t}_{p_i}^{(i)} )^{ \top }$, $i =1 ,\dots, d$.
	
	 \item[3.]Assign $\vect{\alpha} = \hat{\vect{\alpha}} $, $\bfpi_i:=\hat{\bfpi}_i$, $\bfT_i :=\hat{\bfT}_i$, $ \bft_i :=\hat{\bft}_i$, $i =1 ,\dots, d$, and GOTO 1.
	 \end{enumerate} 
	  \State \textit{\textbf{Output}}: Fitted parameters $( \bfpi_1,\dots,\bfpi_d, \bfT_1,\dots, \bfT_d, \vect{\alpha})$.
\end{algorithmic}
\end{algorithm}

\subsection{A class of multivariate SIPH distributions}\label{subsec:MSIPH}
We now proceed to incorporate deterministic time-inhomogeneity into the shared scaling construction. Consider conditionally independent random variables $(Y_1,\dots,Y_d)^{\top}$ given $\Theta=\theta$ by 
\[  Y_i \mid \Theta = \theta \sim \mbox{IPH}(\vect{\pi}_i,\mat{T}_i, \theta \lambda_i) \,, \quad i = 1, \dots, d \,.  \]
Then 
\begin{align*}
	S_{\bfY }(\bfy ) &=\int \P (Y_1>y_1,\dots,Y_d>y_d \mid \Theta=\theta)dF_\Theta (\theta) \\
	&= \int \prod_{i=1}^d \vect{\pi}_i \exp\left({\theta \mat{T}_i h_i^{-1}(y_i)}\right)\vect{e}dF_\Theta (\theta) \\
	&= \int (\vect{\pi}_1\otimes\cdots \otimes \vect{\pi}_d ) \exp\left({\theta (\mat{T}_1 h_1^{-1}(y_1)\oplus\cdots\oplus\mat{T}_d h_d^{-1}(y_d))}\right)\vect{e} dF_\Theta (\theta) \\
	& =   (\vect{\pi}_1\otimes\cdots \otimes \vect{\pi}_d ) \mathcal{L}_\Theta(-(\mat{T}_1 h_1^{-1}(y_1)\oplus\cdots\oplus\mat{T}_d h_d^{-1}(y_d))) \bfe, \:\:y_i\ge0,\:\: i=1,\dots,d,
\end{align*}
and 
\begin{align*}
	f_{\bfY }(\bfy ) 
	& =  \left(\prod_{i =1 } ^ d -\lambda_i(y_i) \right) (\vect{\pi}_1\otimes\cdots \otimes \vect{\pi}_d ) \mathcal{L}_{\Theta}^{(d)}(-(\mat{T}_1 h_1^{-1}(y_1)\oplus\cdots\oplus\mat{T}_d h_d^{-1}(y_d))) \tilde{\bft} \,,
\end{align*}
 where $h^{-1}_i (y) = \int_0^{y} \lambda_i(t) dt$, $i = 1, \dots, d$.
Note that $\bfY$ has representation $(Y_1, \dots ,Y_d)^{\top} = (h_1(Z_1/\Theta), \dots ,h_d(Z_d/\Theta))^{\top}$, which can be seen as follows 
\begin{align*}
	 \P(Y_1 > y_1, \dots ,Y_d>y_d  ) &= \int \P(Y_1 > y_1, \dots ,Y_d>y_d \mid \Theta = \theta ) dF_\Theta(\theta) \\
	& = \int \P(Z_1 > \theta h_1^{-1}(y_d), \dots ,Z_d>\theta h_d^{-1}(y_d) \mid \Theta = \theta ) dF_\Theta(\theta) \\
	& = \int \prod_{i=1}^d \vect{\pi}_i \exp\left({\theta \mat{T}_i h_i^{-1}(y_i)}\right)\vect{e} dF_\Theta(\theta) \\ 
	&  = S_{\bfY }(\bfy ) \,.
\end{align*}

\begin{example}[Positive stable scaling] \rm 

Take  $\Theta$ positive stable with stability parameter $\alpha \in (0, 1]$, then 
\begin{align*}
	S_{\bfY }(\bfy ) = (\vect{\pi}_1\otimes\cdots \otimes \vect{\pi}_d ) \exp\left(-(-\mat{T}_1 h_1^{-1}(y_1)\oplus\cdots\oplus\mat{T}_d h_d^{-1}(y_d))^{\alpha}\right) \bfe \,.
\end{align*}
For the particular case $\lambda_i(y) \equiv \eta_i y^{\eta_i - 1} $, $\eta_i>0$, $i = 1, \dots, d$, we have
\begin{align*}
	S_{\bfY }(\bfy ) = (\vect{\pi}_1\otimes\cdots \otimes \vect{\pi}_d ) \exp\left(-(-\mat{T}_1 y_1^{\eta_1}\oplus\cdots\oplus\mat{T}_d y_d^{\eta_d})^{\alpha}\right) \bfe \,.
\end{align*}
This joint distribution can be seen to be a matrix-parameter version of the multivariate Weibull distribution introduced in \cite{manatunga1999parametric}. 
 
\end{example}

\begin{remark} \rm 
	Covariates can be incorporated into the model by assuming that the intensities are of the form
	 \begin{align*}
	 	\lambda_i(t ; \Theta, \bfX) = \Theta \lambda_i(t) \exp(\vect{\beta}\bfX ), \quad t\ge0, \quad i=1,\dots,d.
	\end{align*}
\end{remark}

\begin{remark}[Shared frailty model]\rm
In the shared frailty model, it is assumed that a group of individuals is conditionally independent given the frailty. In this way, the conditional joint survival function of $\bfY \mid \Theta=\theta$, $\bfY = (Y_1, \dots, Y_d)^{\top}$, is given by 
\begin{align*}
	S_{\bfY | \Theta}(\bfy | \theta) & = \P (Y_1>y_1,\dots,Y_d>y_d \mid \Theta = \theta) \\
	& =\prod_{i = 1 }^{d} \exp(- \theta M_i(y_i)) \\
	& = \exp \left( -\theta \sum_{i =1}^{d} M_i(y_i) \right) \,,
\end{align*}
where $M_i$ are baseline cumulative hazards, $i = 1,\dots, d$. Thus, the joint survival function of $\bfY$ is given by 
\begin{align*}
	S_{\bfY }(\bfy ) = \mathcal{L}_\Theta \left( \sum_{i =1}^{d} M_i(y_i) \right) \,.
\end{align*}
Using that 
\begin{align*}
	M_i(y) = \mathcal{L}_\Theta^{-1}(S_{Y_i}(y)) \,, \quad i = 1, \dots, d ,
\end{align*}
the above joint survival function can be rewritten as 
\begin{align*}
	S_{\bfY }(\bfy ) = \mathcal{L}_\Theta \left( \sum_{i =1}^{d} \mathcal{L}_\Theta^{-1}(S_{Y_i}(y_i)) \right) \,.
\end{align*}
In particular, this means that the survival copula of $\bfY$ is an Archimedean copula. 
Note that the shared frailty model is a particular case of the class of multivariate SIPH distributions introduced here when $p = 1$.
\end{remark}

We now study the dependence structure of multivariate SIPH distributions. When $p = 1$, the survival copula of $\bfY$ is an Archimedean copula. To study the more general case, note that all the transformations presented in Table~\ref{tab:trans} are strictly increasing. This means that the copulas for models based on these intensities are the same as the ones of the models presented in Section~\ref{sec:MCPH}, and thus it is enough to study the later case. 
Define the coefficient of upper tail dependence as
    \begin{align*}
	\lambda_U (\bfY) = \lim_{q \to 1^{-}} {\P(Y_1 > F_{Y_1}^{\leftarrow}(q) \mid Y_2 > F_{Y_2}^{\leftarrow}(q)) } \,.
	\end{align*}
\begin{proposition}
Let    $V := 1/ \Theta$ be regularly varying with index $\alpha>0$. Then
\begin{align*}
	\lambda_U (\bfY) & = \Gamma(\alpha + 1) (\bfpi_1 \otimes \bfpi_2)(-\tilde{\bfT}_1\oplus \tilde{\bfT}_2 )^{-\alpha} \bfe \, ,
\end{align*}
where $\tilde{\bfT}_i := \bfT_i \E(Z_i^\alpha)^{1/\alpha}$, $i =1,2$.
\end{proposition}
\begin{proof}
Given the definition of our model, Proposition 1 of Section 2 in \cite{engelke2018extremal} yields
\begin{align*}
	\lambda_U (\bfY) = \E \left( \min \left( \frac{Z_1^\alpha}{\E(Z_1^\alpha)} , \frac{Z_2^\alpha}{\E(Z_2^\alpha)} \right) \right) \,,
\end{align*}
where $Z_i$ are PH($\bfpi_i,\bfT_i$), and
\begin{align*}
	\E(Z_i^\alpha) = \Gamma(\alpha + 1) \bfpi_i(-{\bfT}_i )^{-\alpha} \bfe \,, \quad i=1,2 \,.
\end{align*}
Moreover, ${Z_i}/{\E(Z_i^\alpha)^{1/\alpha}}$ is PH distributed with the same vector of initial probabilities $\bfpi_i$ and sub-intensity matrix $\tilde{\bfT}_i = \bfT_i \E(Z_i^\alpha)^{1/\alpha} $, $i =1 ,2$. This implies that 
\begin{align*}
	\min \left( \frac{Z_1}{\E(Z_1^\alpha)^{1 / \alpha}} , \frac{Z_2}{\E(Z_2^\alpha)^{1 / \alpha}} \right)  \sim \mbox{PH} (\bfpi_1 \otimes \bfpi_2, \tilde{\bfT}_1\oplus \tilde{\bfT}_2) \,,
\end{align*}
which now yields
\begin{align*}
	\lambda_U (\bfY) & = \Gamma(\alpha + 1) (\bfpi_1 \otimes \bfpi_2)(-\tilde{\bfT}_1\oplus \tilde{\bfT}_2 )^{-\alpha} \bfe \,.
\end{align*}
\end{proof}
 Note that the resulting explicit expression for $\lambda_U$ is in terms of the parameters of the PH components. For instance, when considering $\Theta \sim \mbox{Gamma}(\alpha,1)$, the survival copula of the model can be different from the Clayton copula, for which $\lambda_U = 2^{-\alpha}$. In Figure~\ref{fig:copula}, we take the same value $\alpha = 1$ and plot the implicit copula of two multivariate CPH distributions, one with upper tail dependence coefficient smaller than $2^{-1}$ and the other larger than $2^{-1}$, achieved solely by changing the parameters of the PH components. 
 
\begin{figure}[h]
\centering
\includegraphics[width=0.49\textwidth]{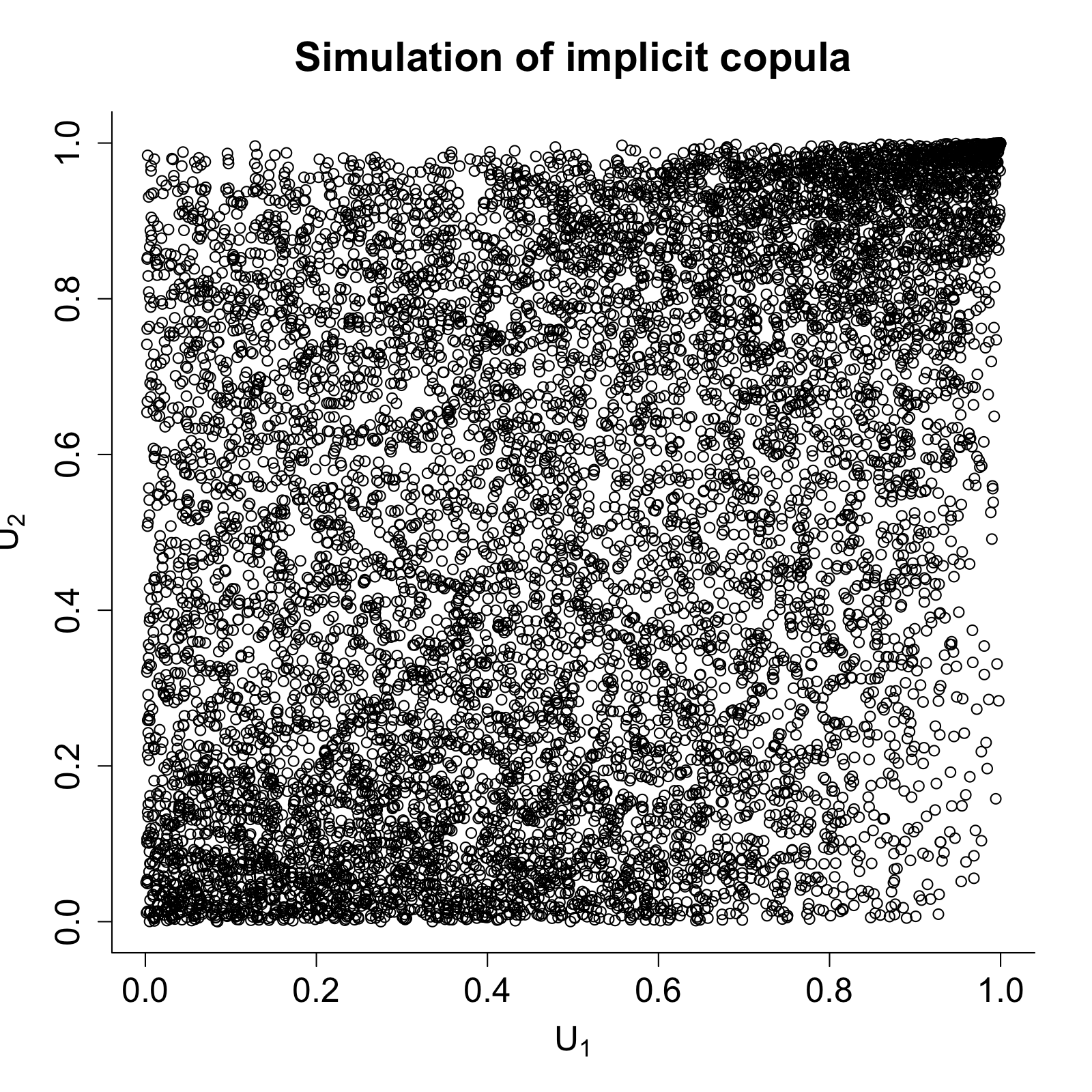}
\includegraphics[width=0.49\textwidth]{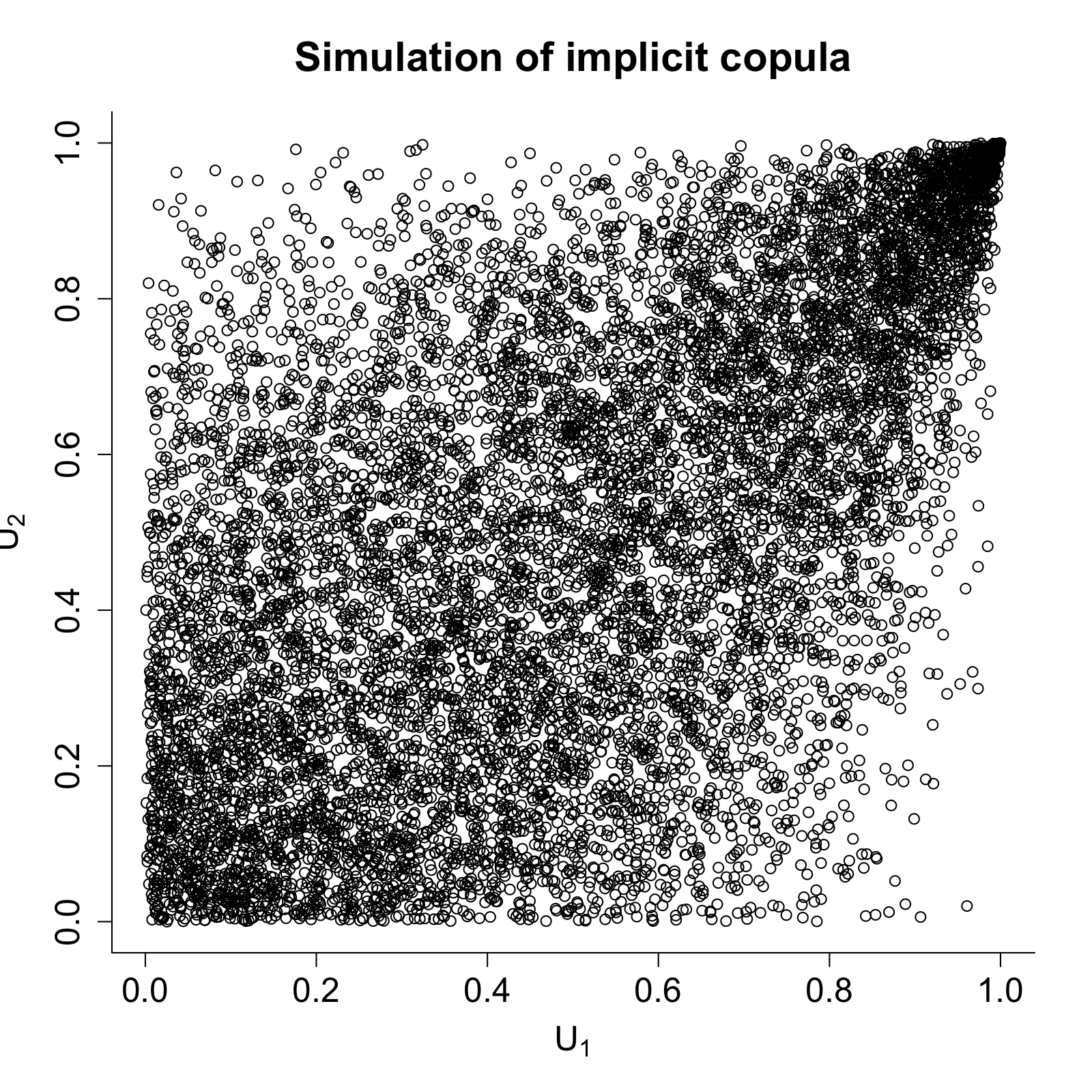}
\caption{Simulation of implicit copulas of multivariate SIPH with $\lambda_U = 0.4128$ (left), and  multivariate SIPH with $\lambda_U = 0.5659$ (right).}
\label{fig:copula}
\end{figure}


\subsection{Parameter estimation: multivariate SIPH distributions}
If we assume that $\lambda_i(\,\cdot\, ; \vect{\eta}_i)$ is a parametric function depending on the vector $\vect{\eta}_i$, $i =1 ,\dots, d$, and let $\vect{\eta} = (\vect{\eta}_1, \dots, \vect{\eta}_d)$. Then we can use that  $(h^{-1}_1(Y_1 ; \vect{\eta}_1), \dots, h^{-1}_d(Y_d;\vect{\eta}_d))^{\top} \eqd (Z_1/\Theta, \dots ,Z_d/\Theta)^{\top} $ to  formulate a generalized EM algorithm for maximum-likelihood estimation, which generalizes Algorithm~\ref{alg:SIPH} to the multivariate case.

 \begin{algorithm}
 \caption{Generalized EM for multivariate SIPH distributions}\label{alg:MSIPH}
 	\begin{algorithmic}
 	\State \textit{\textbf{Input}:} Initialize with $( \bfpi_1, \dots,  \bfpi_d,\bfT_1, \dots, \bfT_d , \vect{\alpha}, \vect{\eta})$.
 	\begin{enumerate} 
 	\item[ 1.] Transform the data into $z_n^{(i)}=h_i^{-1}(y_n^{(i)}; \vect{\eta}_i)  $,  $i=1,\dots,d$, $n=1,\dots,K$, and 
 	apply the E- and M-steps of Algorithm~\ref{alg:EMMCPH} by which we obtain the estimators $( \hat{\bfpi}_1, \dots, \hat{\bfpi}_d,\hat{\bfT}_1, \dots, \hat{\bfT}_d, \hat{\vect{\alpha}})$.
 	
 	\item[ 2.] Compute  
 	\begin{align*}
 		\hat{\vect{\eta}}  & = \argmax_{\vect{\eta}} \sum_{n=1}^{K} \log (f_{\bfY}(\bfy_n; \hat{\bfpi}_1,\dots,\hat{\bfpi}_d,  \hat{\bfT}_1,\dots, \hat{\bfT}_d, \hat{\vect{\alpha}}, \vect{\eta})) \,.
 	\end{align*}
 	
 	\item[ 3.] Assign $(\vect{\pi}_1,\dots,\vect{\pi}_d,\mat{T}_1,\dots,\mat{T}_d,\vect{\alpha},\vect{\eta}) =(\hat{\bfpi}_1,\dots,\hat{\bfpi}_d,  \hat{\bfT}_1,\dots, \hat{\bfT}_d, \hat{\vect{\alpha}}, \hat{\vect{\eta}})$ and GOTO 1. 
 	\end{enumerate} 
 	\State \textit{\textbf{Output}}: Fitted parameters $( \bfpi_1, \dots,  \bfpi_d,\bfT_1, \dots, \bfT_d , \vect{\alpha}, \vect{\eta})$.
 \end{algorithmic}
 \end{algorithm}

\section{Correlated  scaling}\label{sec:correlated}

We now extend the scaling of the sub-intensity matrix of SIPH distributions to the case where we condition on a random vector, the scaling factors being the components of such vector. We consider first the conditionally PH case, i.e. when no deterministic time-transform is present, and a scaling vector $\vect{\Theta} = (\Theta_1, \dots,\Theta_d)^{\top}$ and $\bfY = (Y_1,, \dots, Y_d)^{\top}$ such that the random variables $Y_i$ are conditionally independent given $\vect{\Theta}$ with laws
\begin{align*}
	Y_i \mid \vect{\Theta} = (\theta_1, \dots,\theta_d)^{\top} \sim \mbox{PH}(\bfpi_i, \theta_i\bfT_i)\,, \quad  i = 1, \dots, d \,.
\end{align*}
Then, in full generality, the joint survival function of $\bfY$ is given by 
\begin{align*}
	S_{\bfY}(\bfy) = \int \prod_{i = 1}^{d} \bfpi_i \exp\left({\theta_i \bfT_i y_i}\right) \bfe  dF_{\vect{\Theta}}(\bftheta),\quad y_i\ge0,\:\:i=1,\dots,d.
\end{align*}
Consider the bivariate case. Then, using functional calculus, we have that that the joint survival function takes the explicit form
\begin{align*}
	S_{\bfY}(\bfy) & = \int \bfpi_1 \exp\left({\theta_1 \bfT_1 y_1}\right) \bfe \bfpi_2 \exp\left({\theta_2 \bfT_2 y_2}\right) \bfe  dF_{\vect{\Theta}}(\bftheta) \\
	& = \int (\bfpi_1 \otimes \bfpi_2) \exp\left({\theta_1 \bfT_1 y_1 \oplus \theta_2 \bfT_2 y_2}\right) \bfe  dF_{\vect{\Theta}}(\bftheta) \\
	& = \int (\bfpi_1 \otimes \bfpi_2) \exp\left({\theta_1 \bfT_1 y_1 \otimes \bfI_2 + \bfI_1 \otimes \theta_2 \bfT_2 y_2}\right) \bfe  dF_{\vect{\Theta}}(\bftheta) \\
	& = (\bfpi_1 \otimes \bfpi_2) \mathcal{L}_{\bfTheta}(-\bfT_1 y_1 \otimes \bfI_2 ,  -\bfI_1 \otimes  \bfT_2 y_2) \bfe, \quad y_1,y_2\ge0,
\end{align*}

where $\mathcal{L}_{\vect{\Theta}}$  is  the joint Laplace transform of $\vect{\Theta}$, that is 
\begin{align*}
	\mathcal{L}_{\vect{\Theta}} (u_1,u_2 ) = \E \left( \exp\left({-u_1 \Theta_1 - u_2 \Theta_2  }\right)\right),\quad u_1,\,u_2\ge0.
\end{align*}

Note that $\bfY = (Z_1 /\Theta_1, \dots, Z_d /\Theta_d)^{\top}$, where $Z_i$ are independent $\mbox{PH}(\bfpi_i, \bfT_i)$ distributed random variables, $i = 1, \dots, d$, independent of $\bfTheta$. Indeed,
\begin{align*}
	\P\left( Y_1> y_1, \dots, Y_d > y_d \right) &= \int \P\left( Y_1> y_1, \dots, Y_d > y_d \mid \bfTheta \right) dF_{\vect{\Theta}}(\bftheta) \\
	&= \int \P\left( Z_1> \theta_1 y_1, \dots, Z_d > \theta_d y_d \mid \bfTheta \right) dF_{\vect{\Theta}}(\bftheta) \\
	& = \int \prod_{i = 1}^{d} \bfpi_i \exp\left({\theta_i \bfT_i y_i}\right) \bfe  dF_{\vect{\Theta}}(\bftheta) \\
	& = S_{\bfY}(\bfy) \,.
\end{align*}

\subsection{Parameter estimation: correlated CPH distributions}

The maximum-likelihood estimation of this class of multivariate distributions can be performed via a generalized EM algorithm.  The derivation is done similarly to Algorithm~\ref{alg:EMMCPH} and thus omitted for brevity. Again, for estimation, we assume that $\bfTheta$ belongs to a parametric family depending on the vector $\bfalp$ and denote by $f_{\bfTheta}$ its corresponding joint density. The resulting detailed  routine is provided in Algorithm \ref{alg:EMCorCPH}.

\begin{algorithm}[]
\caption{Generalized EM algorithm for correlated CPH distributions}\label{alg:EMCorCPH}
\begin{algorithmic}
	\State \textit{\textbf{Input}}: Initialize with $( \bfpi_1,\dots,\bfpi_d, \bfT_1,\dots, \bfT_d, \vect{\alpha})$.
	\begin{enumerate} 
	\item[ 1.]\textit{E-step:} For each $i =1 ,\dots, d$, calculate

\begin{align*}
	& \E \left( B_{k}^i \mid \tilde{\bfY} = \tilde{\bfy} \right) 
	=\sum_{n=1}^{K} \int \frac{  \pi_{k}^{(i)} \bfe^{\top}_{k} \exp({\theta_i \bfT_i y_{n}^{(i)} }) \theta_i \bft_i \prod_{j \neq i }^{d} \bfpi_j \exp({ \theta_j \bfT_j y_{n}^{(j)}}) \theta_j \bft_j }{f_{\bfY}(\bfy_n)} \, f_{\bfTheta}(\bftheta) d\bftheta   \\[3mm]
	&\E \left( \Theta_i Z_{k}^i \mid \tilde{\bfY} = \tilde{\bfy} \right) \\
	&  = \sum_{n=1}^{K} \int  \theta_i \frac{ \int^{y_n^{(i)}}_{0}  \bfe^{\top}_{k} \exp({\theta_i \bfT_i(y_n^{(i)}-u)})\theta_i \bft_i \bfpi_i \exp({\theta_i \bfT_i u })\bfe_{k}   du \prod_{j \neq i }^{d} \bfpi_j \exp({ \theta_j \bfT_j y_n^{(j)}}) \theta_j \bft_j }{f_{\bfY}(\bfy_n)}  f_{\bfTheta}(\bftheta) d\bftheta  \\[3mm]
	& \E \left( N_{kl}^i \mid \tilde{\bfY} = \tilde{\bfy}\right) \\
	&  = \sum_{n=1}^{K} \int \theta_i t_{kl}^{(i)} \frac{ \int^{y_n^{(i)}}_{0}  \bfe^{\top}_{l} \exp({\theta_i \bfT_i(y_n^{(i)}-u)})\theta_i \bft_i \bfpi_i \exp({\theta_i \bfT_i u })\bfe_{k}   du \prod_{j \neq i }^{d} \bfpi_j \exp({ \theta_j \bfT_j y_n^{(j)}}) \theta_j \bft_j }{f_{\bfY}(\bfy_n)}  f_{\bfTheta}(\bftheta) d\bftheta \\[3mm]
	& \E \left( N_{k}^i \mid \tilde{\bfY} = \tilde{\bfy} \right)
	 = \sum_{n=1}^{K} \int \theta_i t_{k}^{(i)}   \frac{\bfpi_i \exp({\theta_i \bfT_i y_n^{(i)} })\bfe_{k}   \prod_{j \neq i }^{d} \bfpi_j \exp({ \theta_j \bfT_j y_n^{(j)}}) \theta_j \bft_j  }{f_{\bfY}(\bfy_n)} f_{\bfTheta}(\bftheta) d\bftheta
\end{align*}

	\item[ 2.]\textit{M-step:} Let
	\begin{align*}
	\hat{\vect{\alpha}} &= \argmax_{\vect{\alpha}} \E \left( \log(f_{\bfTheta}(\bfTheta; \vect{\alpha}) ) \mid \tilde{\bfY} = \tilde{\bfy}  \right) \\
	 &= \argmax_{\vect{\alpha}} \sum_{n=1}^{K} \int \log(f_{\bfTheta}(\bftheta; \vect{\alpha}) )   \frac{ \prod_{ i = 1}^{d} \bfpi_i \exp({ \theta_i \bfT_i y_n^{(i)}}) \theta_i \bft_i  \, }{f_{\bfY}(\bfy_n)}f_{\bfTheta}(\bftheta) d\bftheta
	\end{align*}
	and 
	\begin{align*}
		&\hat{\pi}_{k}^{(i)} = \frac{\E\left( B_{k}^i \mid \tilde{\bfY} = \tilde{\bfy}  \right)}{K}  
		\,, \quad
		\hat{t}_{kl}^{(i)} = \frac {\mathlarger{ \E\left( N_{kl}^i \mid \tilde{\bfY} = \tilde{\bfy}  \right) }}{\mathlarger{ \E \left(\Theta_i Z_{k}^i \mid \tilde{\bfY} = \tilde{\bfy}  \right) }}
		\,, \quad
		\hat{t}_{k}^{(i)} = \frac {\mathlarger{ \E\left( N_{k}^i \mid \tilde{\bfY} = \tilde{\bfy}  \right) }} {\mathlarger{ \E\left( \Theta_i Z_{k}^i \mid \tilde{\bfY} = \tilde{\bfy}  \right)}}
		\,, \\
		&\hat{t}_{kk}^{(i)} = -\sum_{l \neq k} \hat{t}_{kl}^{(i)} -\hat{t}_{k}^{(i)} \,, \quad i = 1, \dots, d\,.
	\end{align*}
	Let $\hat{\bfpi}_i = ( \hat{\pi}_{1}^{(i)}, \ldots , \hat{\pi}_{p_i}^{(i)} )$, $\hat{\bfT}_i = \{ \hat{t}_{kl}^{(i)} \}_{ k, l = 1, \ldots, p_i}$, and $\hat{\bft}_i = ( \hat{t}_{1}^{(i)}, \ldots, \hat{t}_{p_i}^{(i)} )^{ \top }$, $i =1 ,\dots, d$.
	
	\item[ 3.] Assign $\vect{\alpha} = \hat{\vect{\alpha}} $, $\bfpi_i:=\hat{\bfpi}_i$, $\bfT_i :=\hat{\bfT}_i$, $ \bft_i :=\hat{\bft}_i$, $i =1 ,\dots, d$, and GOTO 1.
	\end{enumerate} 
	\State \textit{\textbf{Output}}: Fitted parameters $( \bfpi_1,\dots,\bfpi_d, \bfT_1,\dots, \bfT_d, \vect{\alpha})$.
\end{algorithmic}
\end{algorithm}

\begin{remark} \rm 
	This algorithm suffers from the curse of dimensionality. The integrals above must  typically be computed numerically, given that explicit expressions are not available. Thus, the number of summands needed for the approximation increases rapidly with the dimension. It is also important  to mention that correlated frailty models are typically employed only in the bivariate case. In such a case, the above algorithm is computationally feasible, thus its relevance. 
\end{remark}

\subsection{Correlated SIPH distributions}
We now introduce an analogous model to the correlated frailty model based on IPH distributions, effectively the most general of our models.
Consider a multivariate random scaling component $\vect{\Theta} = (\Theta_1, \dots,\Theta_d)^{\top}$ and $\bfY = (Y_1, \dots,Y_d)^{\top}$, both in in $\mathbb{R}_+^d$, such that $Y_i$ are conditionally independent given $\vect{\Theta}$ with conditional distribution
\begin{align*}
	Y_i \mid \vect{\Theta} = (\theta_1, \dots, \theta_d)^{\top} \sim \mbox{IPH}(\bfpi_i, \bfT_i, \theta_i \lambda_i)\,, \quad i = 1, \dots, d \, .
\end{align*}
The joint survival function of $\bfY$ is then given by 
\begin{align*}
	S_{\bfY}(\bfy) = \int \prod_{i = 1}^{d} \bfpi_i \exp\left({\theta_i \bfT_i h_{i}^{-1}(y_i)}\right) \bfe  dF_{\vect{\Theta}}(\bftheta), \quad y_i\ge0,\:\:i=1,\dots,d.
\end{align*}

In the bivariate case, we have by simple calculations (using functional calculus) the explicit expression 
\begin{align*}
	S_{\bfY}(\bfy) 
	& = (\bfpi_1 \otimes \bfpi_2) \mathcal{L}_{\bfTheta}(-\bfT_1 h_1^{-1}(y_1) \otimes \bfI_2 ,  -\bfI_1 \otimes  \bfT_2 h_2^{-1}(y_2) ) \bfe, \quad y_1,\,y_2\ge0.
\end{align*}
Note that an alternative representation for $\bfY$ is $\bfY = ( h_1 (Z_1 / \Theta_1), \dots, h_d (Z_d / \Theta_d))^{\top}$, where $Z_i$ are independent PH distributed random variables independent of $\bfTheta$. The proof is akin to those of previous sections.

Now we consider a specific example with explicit joint density, namely the correlated Gamma case.
\begin{example}[Correlated Gamma scaling] \rm 
Inspired by \cite{yashin1995correlated}, we
consider $\vect{\Theta} = (\Theta_1, \Theta_2)^{\top}$ such that 
\begin{align*}
	& \Theta_1 = \frac{\eta_0}{\eta_1} W_0 + W_1 \\
	& \Theta_2 = \frac{\eta_0}{\eta_2} W_0 + W_2 \,,
\end{align*}
where $W_i \sim \mbox{Gamma}(\kappa_i, \eta_i)$, $\kappa_i,\eta_i>0$, $i = 0,1,2$, are independent.  
Then we have that 
\begin{align*}
	&\E\left( \exp(- u_1 \Theta_1 -u_2 \Theta_2 )\right) \\
	&\quad = \E\left( \exp\left(- \left(u_1 \frac{\eta_0}{\eta_1} + u_2 \frac{\eta_0}{\eta_2} \right) W_0 -u_1 W_1 -u_2 W_2 \right)\right) \\
	& \quad= \left( 1 + \left(\frac{u_1}{\eta_1} + \frac{u_2}{\eta_2}\right) \right)^{-\kappa_0} \left( 1 + \frac{u_1}{\eta_1} \right)^{-\kappa_1} \left( 1 + \frac{u_2}{\eta_2} \right)^{-\kappa_2}, \:\: u_1,\,u_2\ge0.
\end{align*}
This yields 
\begin{align*}
	 S_{\bfY} (y_1, y_2) 
	 &  = (\bfpi_1 \otimes \bfpi_2 ) \left( \bfI  -\left(\frac{h^{-1}_1(y_1)}{\eta_1} \bfT_1\right) \oplus \left(\frac{h^{-1}_2(y_2)}{\eta_2} \bfT_2\right)  \right)^{-\kappa_0} \\
	 & \quad \cdot\left(\bfI - \left( \frac{h^{-1}_1(y_1)}{\eta_1}\bfT_1\right)\otimes \bfI_2\right)^{-\kappa_1}  \left( \bfI - \bfI_1 \otimes \left(\frac{h^{-1}_2(y_2)}{\eta_2} \bfT_2\right) \right)^{-\kappa_2} \bfe \,.
\end{align*}

One typically sets $\eta_1 = \kappa_0 + \kappa_1 $ and $\eta_2 = \kappa_0 + \kappa_2 $. In this way $\E(\Theta_1) = \E(\Theta_2) = 1$, $\mbox{Var}(\Theta_1) = \eta_1^{-1}$, $\mbox{Var}(\Theta_2) = \eta_2^{-1}$ and  $\mbox{Corr}(\Theta_1, \Theta_2) = \kappa_0 / \sqrt{(\kappa_0 + \kappa_1 )(\kappa_0 + \kappa_2)}$.
\end{example}

%
%
%
%
%
\begin{remark}[Estimation]\rm 
	 Maximum-likelihood estimation can be performed via a modified EM algorithm, which is in the same form as Algorithm~\ref{alg:MSIPH} with the only change in step 1, where we now employ Algorithm~\ref{alg:EMMCPH}. We omit further details. 
\end{remark}


\begin{remark}[Correlated frailty] \rm
	The correlated frailty model assumes that the frailties of individuals are correlated and not necessarily shared. More specifically, in a bivariate correlated frailty model, the conditional joint density of $\bfY \mid \bfTheta = \bftheta$ is 
	\begin{align*}
		S_{\bfY | \bfTheta}(\bfy | \bftheta) = \exp(-\theta_1 M_1(y_1)) \exp(-\theta_2 M_2(y_2)) \,.
	\end{align*}
	In this way, the joint survival function of $\bfY $ is given by
	\begin{align*}
		S_{\bfY}(\bfy) = \mathcal{L}_{\bfTheta} (M_1(y_1), M_2(y_2)) \,.
	\end{align*}
	This is indeed a particular case of the correlated intensities model introduced in the present section when $p = 1$. 
\end{remark}

\section{Numerical illustrations}\label{sec:Numexamples}

In this section, we present some numerical illustrations of practical relevance.
In the first example, we test the performance of Algorithm~\ref{alg:MML} for the estimation of matrix Mittag-Leffler distributions in a simulation study.
In the second example, we consider the fitting of a SIPH distribution to a theoretical given distribution. 
In the third example, we fit a SIPH to a real-life insurance data set. 
As a final example, we perform a simulation study for a multivariate CPH distribution. 
In all cases, we ran the generalized EM algorithms until the changes in the successive log-likelihoods became negligible.

\subsection{Matrix Mittag-Leffler distributions}
We generated an i.i.d.\ sample of size $1,000$ from a matrix Mittag-Leffler distribution of 4 phases with parameters
 \begin{gather*}  
	{\bfpi}=\left(
	0.2, \,0.8,\, 0,\, 0\right)\,, \\ 
	\vect{T}=\left( \begin{array}{cccc}
	-2 &  0 & 2 & 0 \\
	5 & -8 & 0  & 3 \\
	0 & 0 & -1  & 0.5 \\
	0 & 0 & 0  & -4 
	\end{array} \right) \,, \\ 
	\alpha= 0.8 \,.
\end{gather*}
We then fitted a matrix Mittag-Leffler distribution with the same number of phases to the resulting sample using Algorithm~\ref{alg:MML}, obtaining the following parameters:
\begin{gather*}  
	\hat{\bfpi}=\left(
	0, \,0.0381,\, 0.8481, \, 0.1139 \right)\,, \\ 
	\hat{\bfT}=\left( \begin{array}{cccc}
	-3.4286 & 0.1942 & 0.0495 & 0.5393\\
	0.6080 & -1.2013 & 0.0184 & 0.0084  \\
	2.4001 & 2.1178 & -4.7794 & 0.2615  \\
	0.3800 & 0.2744 & 0.3870 & -1.0648 
	\end{array} \right) \,, \\ 
	\hat{\alpha}= 0.7928\,.
\end{gather*}

Observe that we can somewhat retrieve the parameters by keeping in mind possible permutation of states (since their labels are not relevant). Figure~\ref{fig:mml} shows that the algorithm recovers the structure of the data. Moreover, note that $\hat{\alpha}= 0.7928$, which determines the heaviness of the tail, is close to the original value $\alpha = 0.8$.
As further evidence of the quality of the fit, we have that the log-likelihood of the fitted model is $-1,769.596$, while using the original distribution parameters and structure, we obtain $-1,773.453$.
\begin{figure}[h]
\centering
\includegraphics[width=0.49\textwidth]{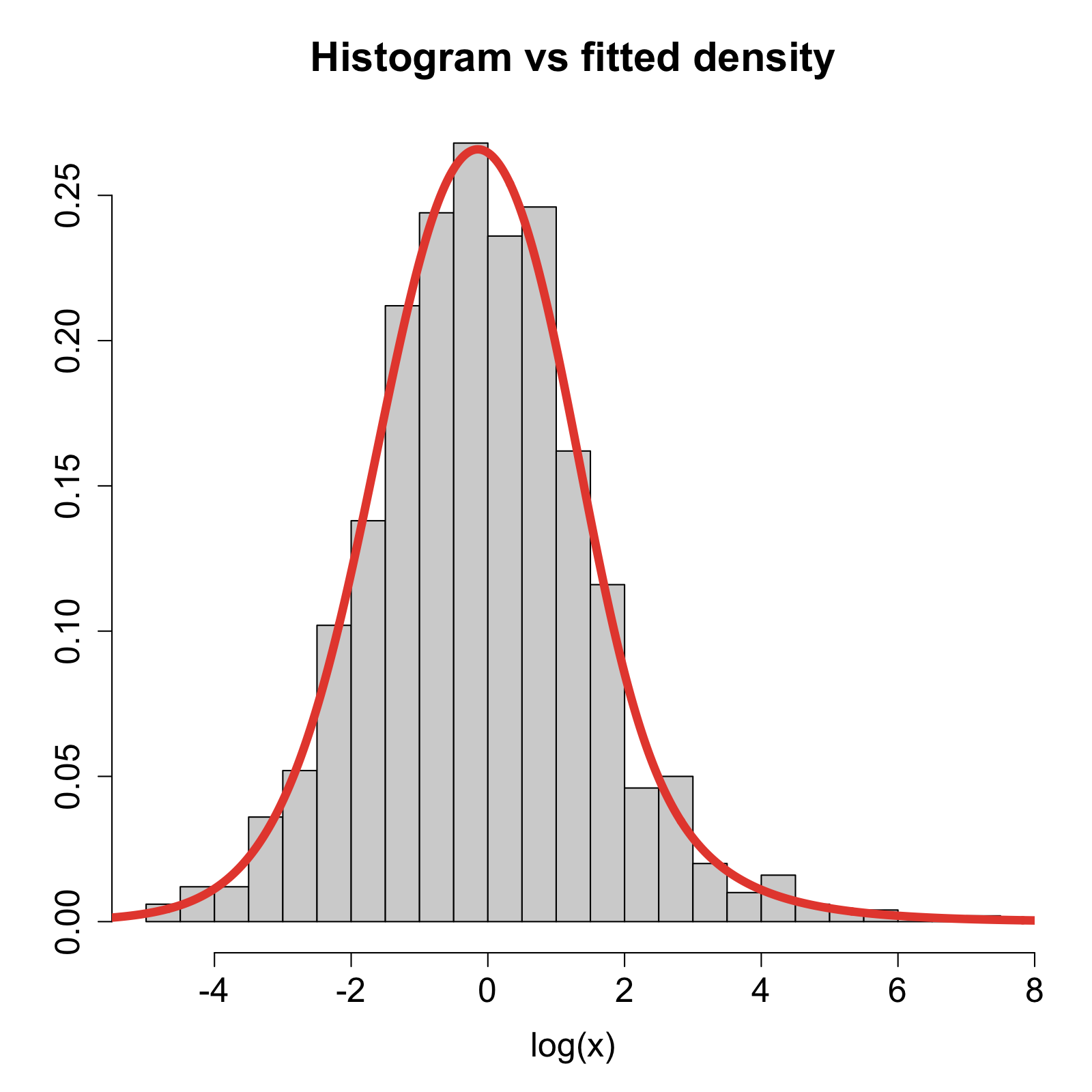}
\includegraphics[width=0.49\textwidth]{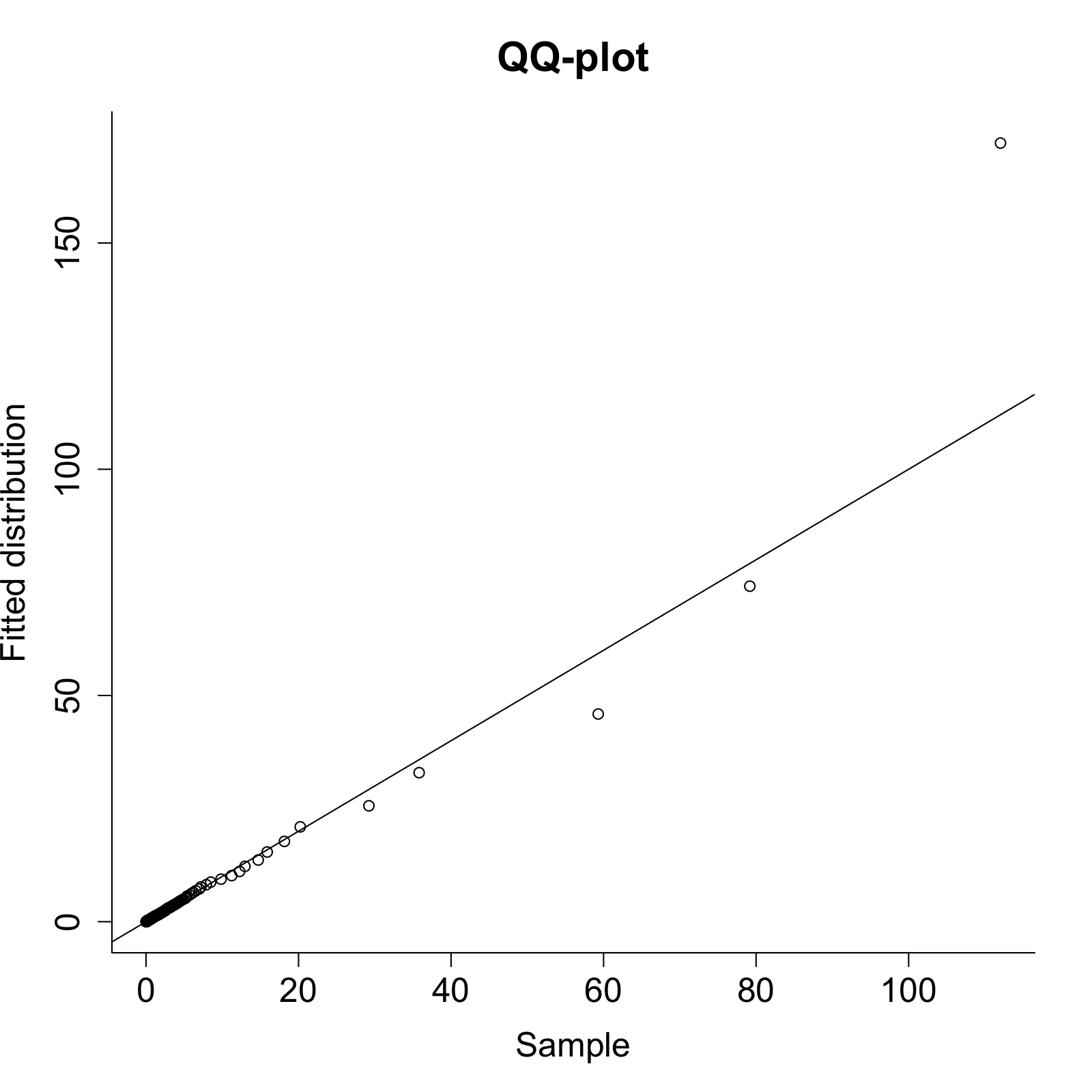}
\caption{Histogram of log-simulated data versus density of the fitted matrix Mittag-Leffler model (left), and corresponding QQ-plot (right).}
\label{fig:mml}
\end{figure}


\subsection{Matrix-Weibull} Algorithm~\ref{alg:SIPH} can be easily modified to approximate given theoretical distributions. As in the PH case (\cite{asmussen1996fitting}), the idea consists of considering sequences of empirical distributions with increasing sample size. 
For instance, if we denote by $g$ the theoretical given density that we want to approximate, in step 1, we have that as $K\to \infty$,
\begin{align*}
	\hat{\pi}_k &= \frac{1}{K} \sum_{n = 1}^{K} \int \frac{\pi_k \bfe_k^\top \exp({\theta\bfT h^{-1}(y_n)}) \theta\bft}{\bfpi \exp({\theta\bfT h^{-1}(y_n)}) \theta\bft}  f_\Theta(\theta) d\theta \\
	&\to \int \int \frac{\pi_k \bfe_k^\top \exp({\theta\bfT h^{-1}(y)}) \theta\bft}{\bfpi \exp({\theta\bfT h^{-1}(y)})\theta \bft} f_\Theta(\theta) d\theta g(y)  dy \,.
\end{align*}
The rest of the formulas in step 1\ are adapted through the same limit. Regarding step 2, we have 
\begin{align*}
 		\hat{\vect{\eta}}  & \to \argmax_{\vect{\eta}} \int \log (f_{Y}(y; \hat{\bfpi}, \hat{\bfT}, \hat{\vect{\alpha}}, \vect{\eta} )) g(y)dy \,.
 \end{align*}

As a concrete example, we consider a Matrix-Weibull distribution (as introduced in \cite{albrecher2019inhomogeneous}, having no random scaling component)
with density function
\begin{align*}
	g(y) = \bfpi \exp(\vect{S} y^{\beta}) \vect{s} \beta y^{\beta -1} \,, \quad y>0 \,,
\end{align*}
and parameters
\begin{gather*}  
	{\bfpi}=\left(
	0.5, \,0.3,\, 0.2 \right)\,, \\ 
	\vect{S}=\left( \begin{array}{ccc}
	-1 &  1 & 0 \\
	0 & -2 & 1  \\
	0 & 0 & -5  
	\end{array} \right) \,, \\ 
	{\beta}= 2 \,.
\end{gather*}

Then we fitted a SIPH distribution of 3 phases with baseline intensity $\lambda(y) = \eta y^{\eta -1}$, $\eta >0$, and positive stable scaling. The fitted parameters are the following
\begin{gather*}  
	\hat{\bfpi}=\left(
	0.1876, \,0.3037,\, 0.5086 \right)\,, \\ 
	\hat{\bfT}=\left( \begin{array}{ccc}
	-1.9843 & 1.2605 &  0.5706\\
	 0.0133 & -1.2985 & 0.1584  \\
	2.3573 & 0.9338 & -5.2052 
	\end{array} \right) \,, \\ 
	\hat{\alpha}= 0.9146\,, \quad \hat{\eta}= 2.1723 \,.
\end{gather*}
The quality of the approximation is supported by Figure~\ref{fig:mweibull}, which shows that we recover the shape of the original distribution. Moreover, the product $\hat{\alpha}\hat{\eta} =  1.9867$, which determines the heaviness of the tail, can be compared with $\beta = 2$ for the given theoretical model.
\begin{figure}[h]
\centering
\includegraphics[width=0.49\textwidth]{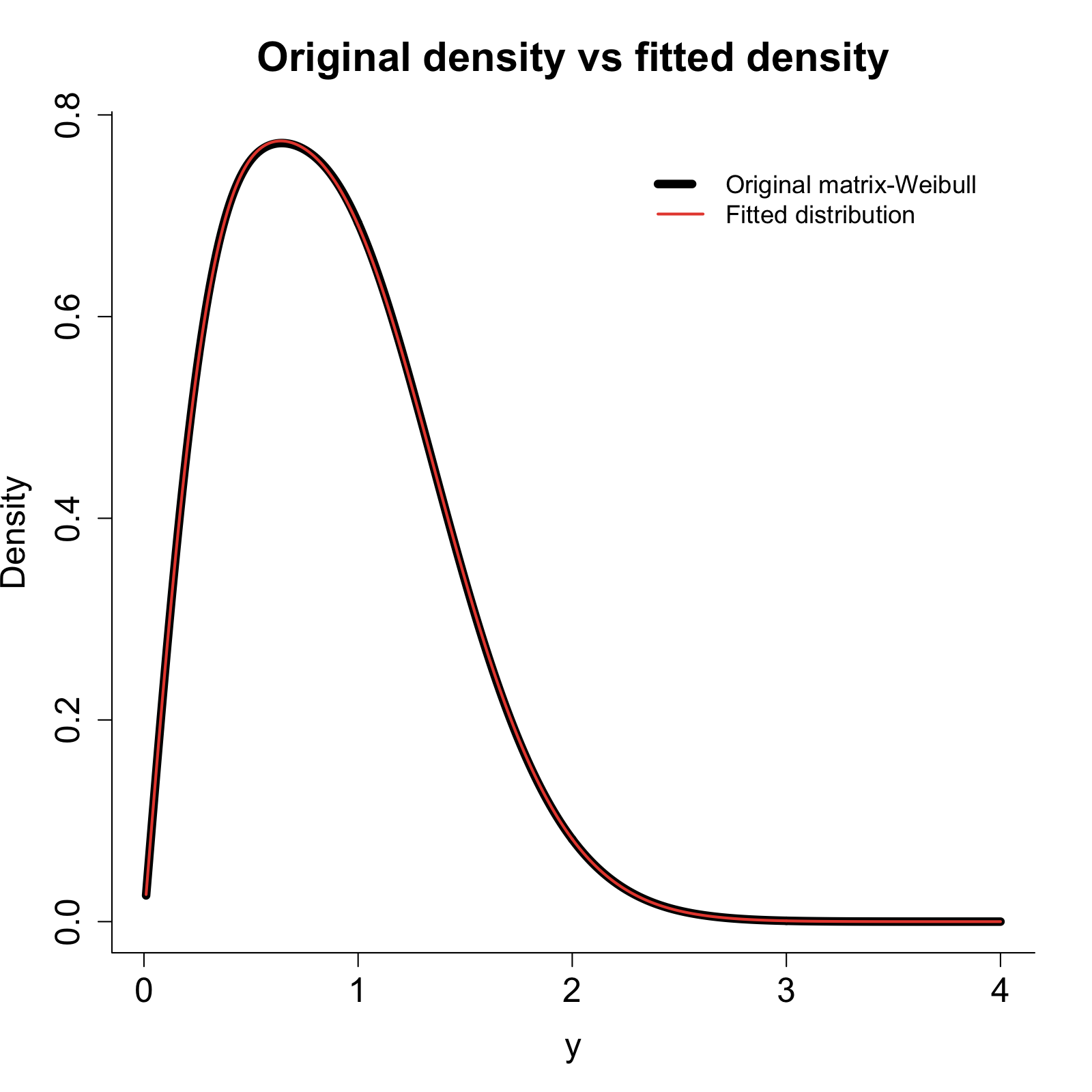}
\includegraphics[width=0.49\textwidth]{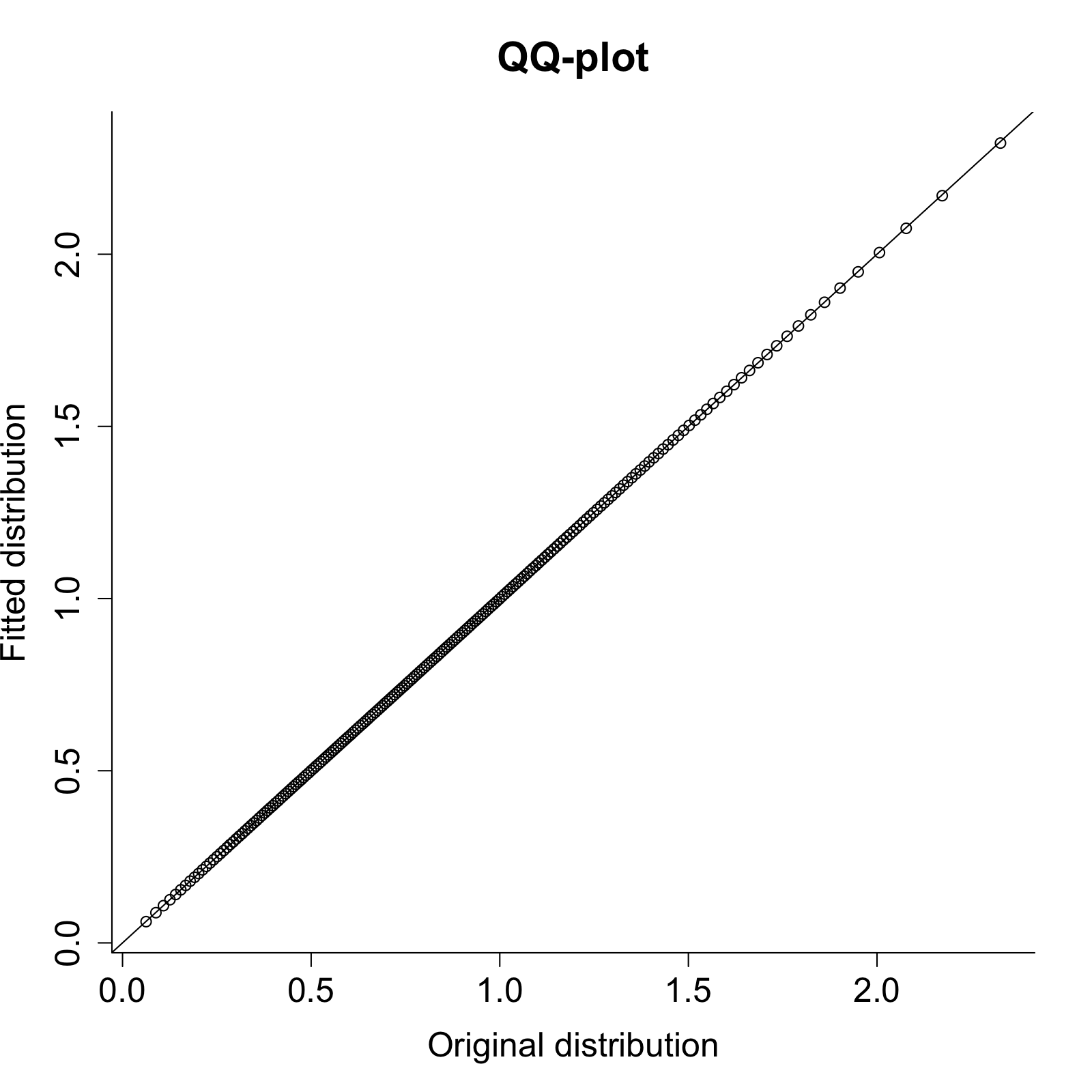}
\caption{Density of the original matrix-Weibull versus density of the fitted SIPH (left), and corresponding QQ-plot (right).}
\label{fig:mweibull}
\end{figure}

\subsection{Real-life data}
The Gamma-Gompertz frailty model is commonly employed for modeling human mortality at old ages (see, e.g., \cite{missov2013gamma,vaupel1979impact}). In the present example, we propose using SIPH distributions with Gamma scaling and Gompertz baseline intensity for modeling this type of data. 

As a concrete case of study, we consider the lifetimes of the Swedish population that die in the year $2011$ between ages $50$-$100$. This data was obtained from the Human Mortality Database (HMD). We add covariate information by considering a separation between females ($X=1$) and males ($X=0$) in the population. Then we fitted a SIPH distribution of $4$ phases with general Coxian structure in the PH component.
The estimated parameters are 
\begin{gather*}  
	\hat{\bfpi}=\left(
	 0.2097, \,0.1572,\, 0.3135,\,0.3196 \right)\,, \\ 
	\hat{\bfT}=\left( \begin{array}{cccc}
	-0.0022 & 0.0004 & 0 & 0 \\
	 0 & -1.1003 & 1.1003 & 0  \\
	0 & 0 & -0.6730 & 0.6730\\
	0 & 0 & 0 & 0.0001
	\end{array} \right) \,, \\ 
	\hat{\alpha}= 5.803\,, \quad \hat{\eta}= 0.1663 \,, \quad \hat{\beta}= -0.5389 \,.
\end{gather*}

Figure~\ref{fig:mGG} shows that the fitted distribution provides a reasonable model for both groups. If an even closer fit is sought, other parameters of the model need to be regressed as well.

\begin{figure}[h]
\centering
\includegraphics[width=0.49\textwidth]{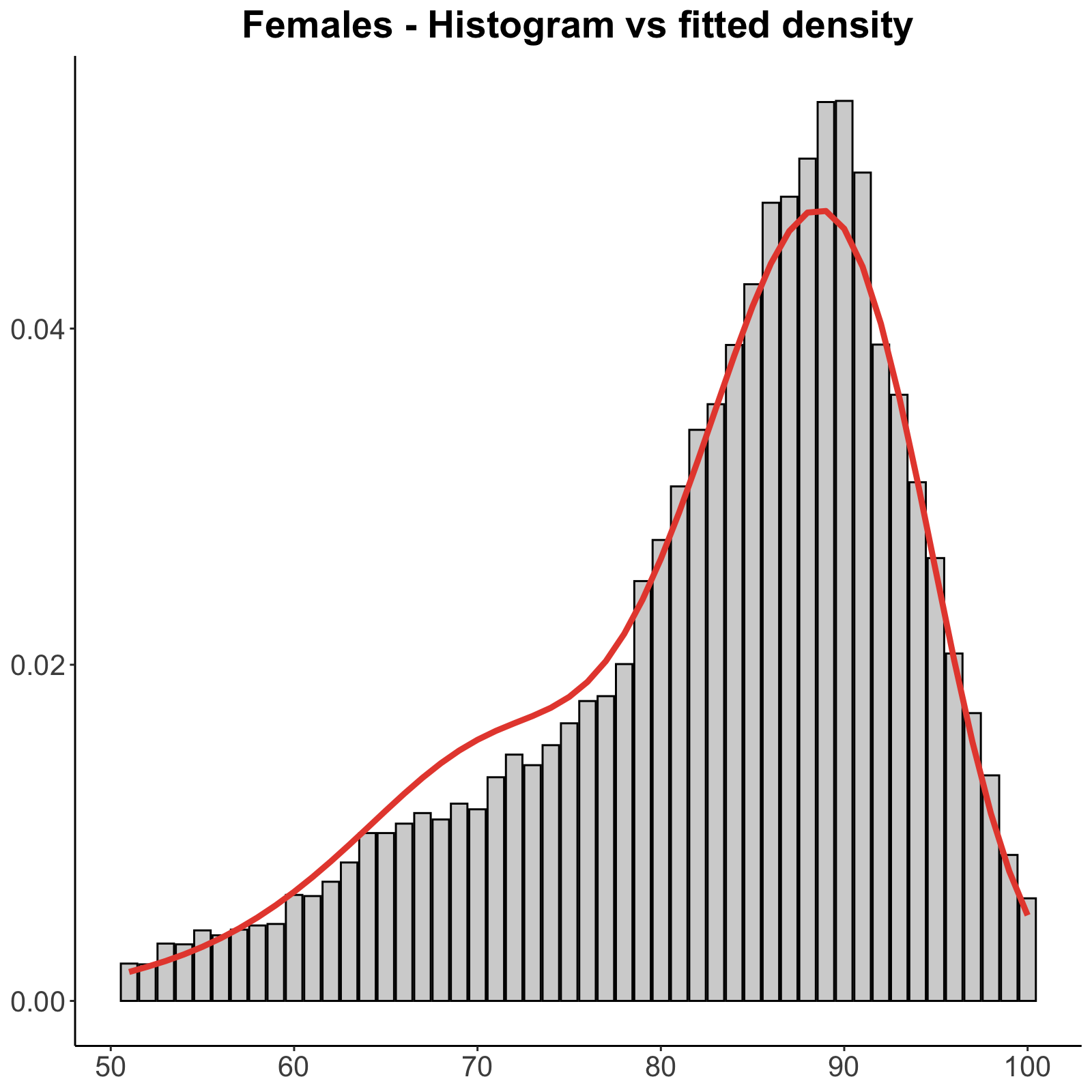}
\includegraphics[width=0.49\textwidth]{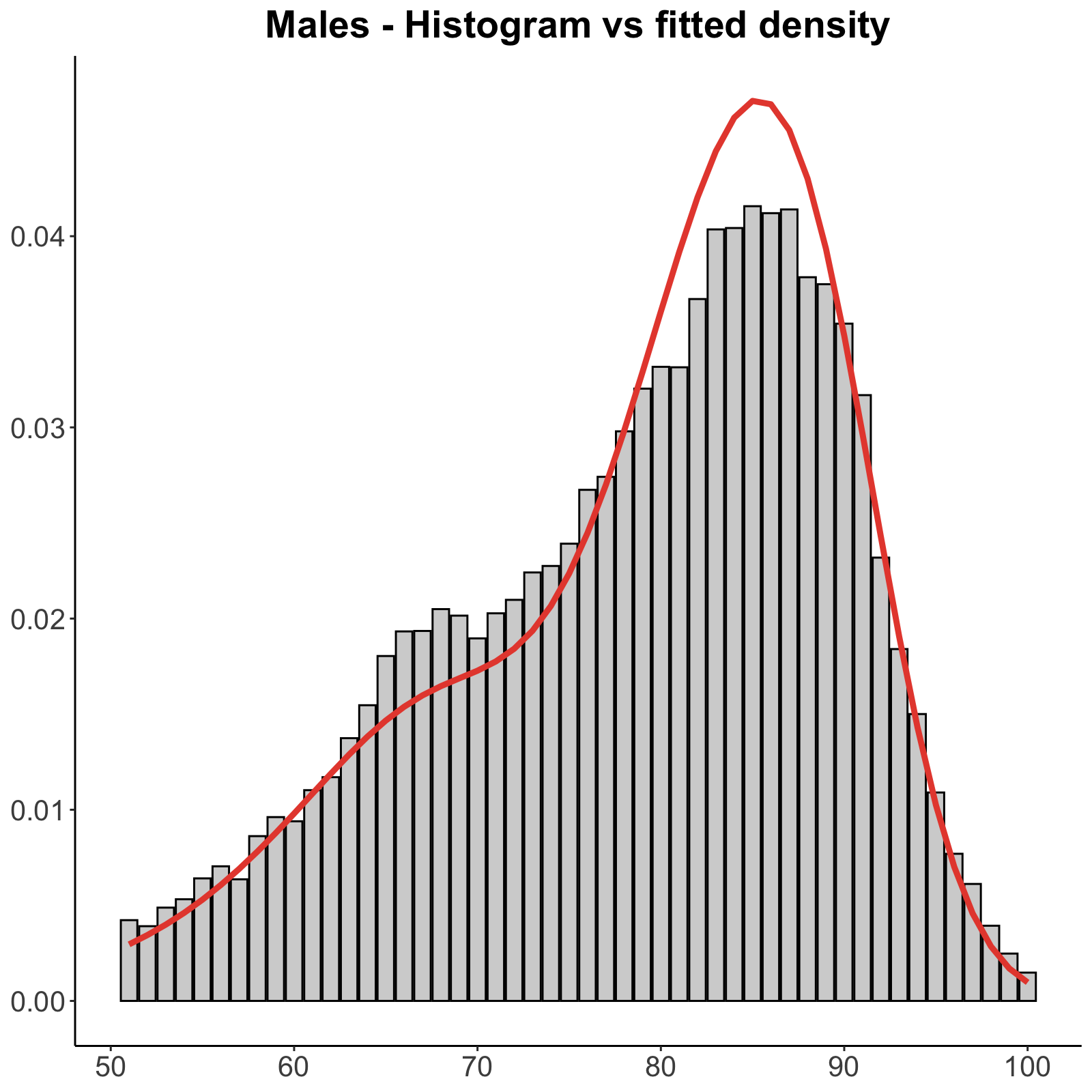}
\caption{Histogram of lifetimes of the Swedish female population that died in 2011 at ages 50 to 100 versus density of the fitted SIPH (left), and corresponding plot for the male population(right).}
\label{fig:mGG}
\end{figure}

\subsection{Multivariate example} We generated an i.i.d.\ sample of size $2, 500$ from a bivariate CPH distribution with parameters
\begin{gather*}  
	{\bfpi_1}=\left(
	1, \,0,\, 0 \right)\,, \\ 
	{\bfT_1}=\left( \begin{array}{ccc}
	-0.5 &  0.2 & 0 \\
	0 & -1 & 0.5  \\
	0 & 0 & -2  
	\end{array} \right) \,, \\
	{\bfpi_2}=\left(
	0.5, \,0.5 \right)\,, \\ 
	{\bfT_2}=\left( \begin{array}{cc}
	-0.1 &  0  \\
	0 & -1 
	\end{array} \right) \,, 
\end{gather*}
and Gamma scaling with $\alpha = 1.5 $. Note that the upper tail dependence coefficient of the theoretical model is $\lambda_U  = 0.2765$, while the empirical estimator of the sample is $\hat{\lambda}_U  = 0.28$.
Then we fitted a bivariate CPH model of same dimensions using Algorithm~\ref{alg:EMMCPH} obtaining the parameters
\begin{gather*}  
	\hat{\bfpi_1}=\left(
	 0.3268, \,0.2124,\, 0.4608\right)\,, \\ 
	\hat{\bfT_1}=\left( \begin{array}{ccc}
	-2.0252 & 1.0067 & 0.9015 \\
	0.0334 & -1.0061 & 0.3753 \\
	0.9293  & 0.6818 & -1.7945 
	\end{array} \right) \,, \\ 
	\hat{\bfpi_2}=\left(
	0.884, \,0.116 \right)\,, \\ 
	\hat{\bfT_2}=\left( \begin{array}{cc}
	-0.8978 & 0.3046  \\
	0.1501 & -0.1546 
	\end{array} \right) \,, \\
	\hat{\alpha} = 1.5874 \,.
\end{gather*}
Figure~\ref{fig:mcph_mar} shows that we recover the structure of both marginals. Regarding the dependence structure, this is supported by Figure~\ref{fig:mcph_cont}, where we offer some contour plots. Moreover, note that the parameter $\alpha$ that determines the heaviness of the tails of the marginals is close to the original model and that the coefficient of upper tail dependence $ \lambda_U  = 0.254$ is close to the original (and sample) one. Finally, note that the original model's log-likelihood is $-11,753.27$, compared with $-11,752.45$ for the fitted model.  

\begin{figure}[h]
\centering
\includegraphics[width=0.49\textwidth]{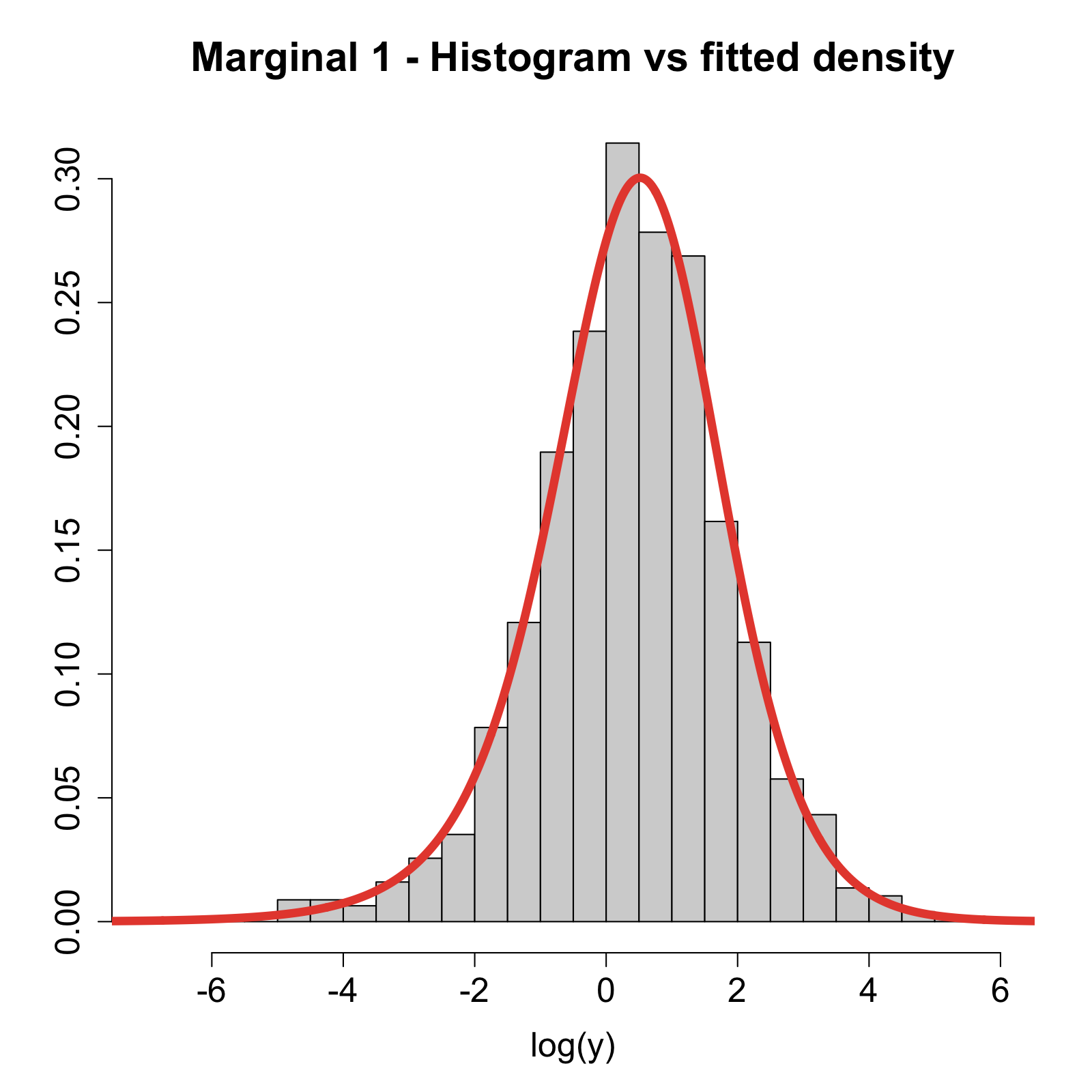}
\includegraphics[width=0.49\textwidth]{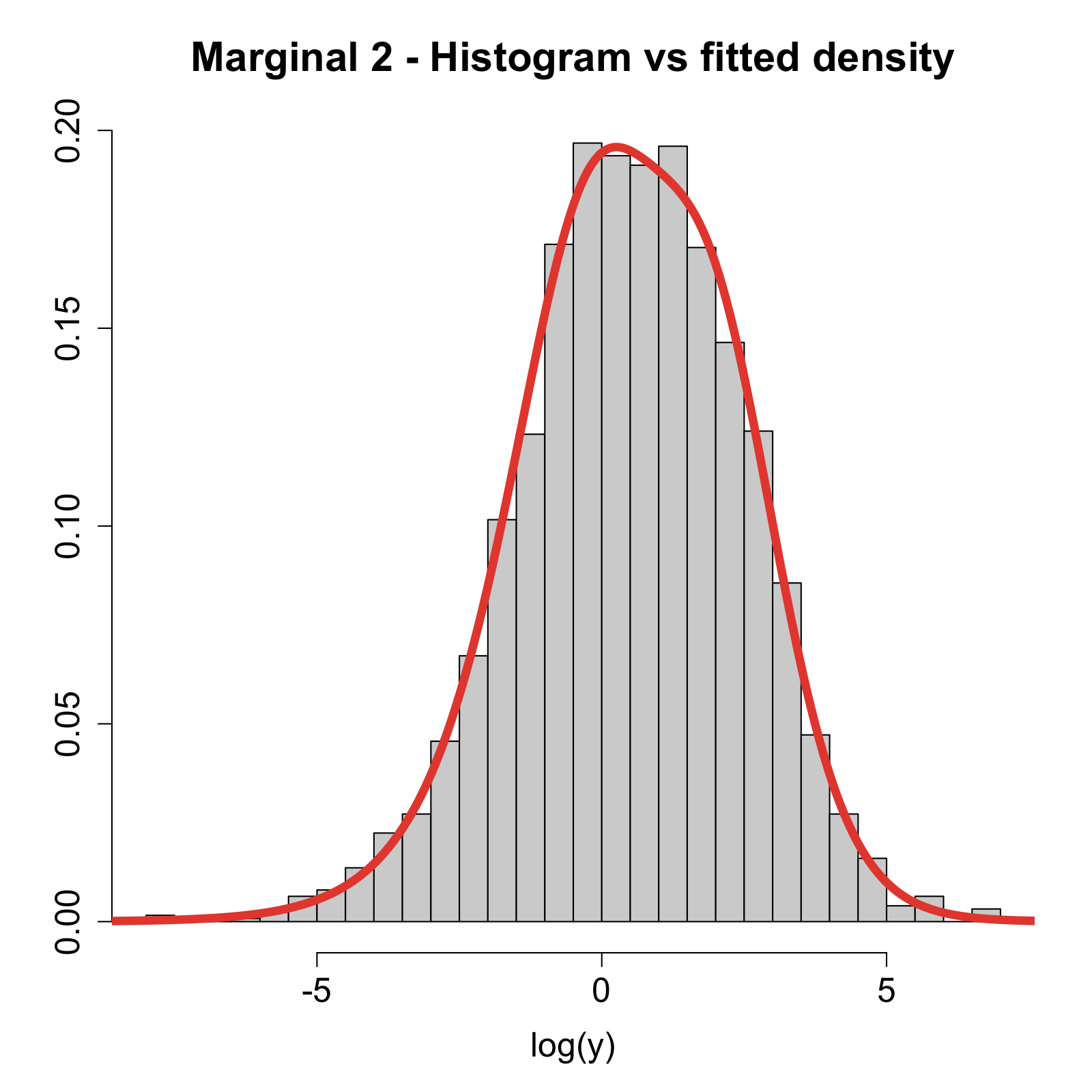}
\caption{Histograms of log-simulated data versus densities of the fitted distribution.}
\label{fig:mcph_mar}
\end{figure}

\begin{figure}[h]
\centering
\includegraphics[width=0.325\textwidth]{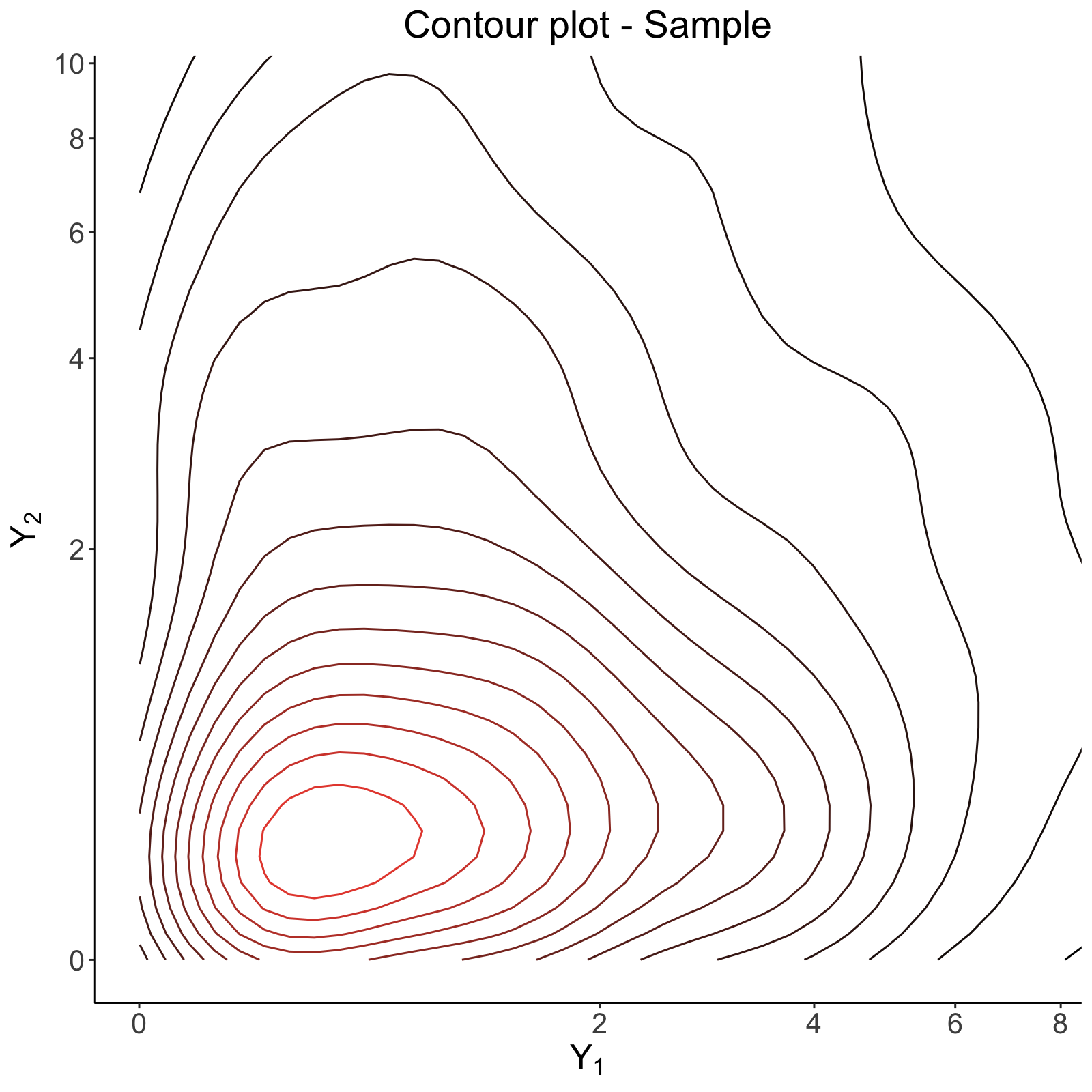}
\includegraphics[width=0.325\textwidth]{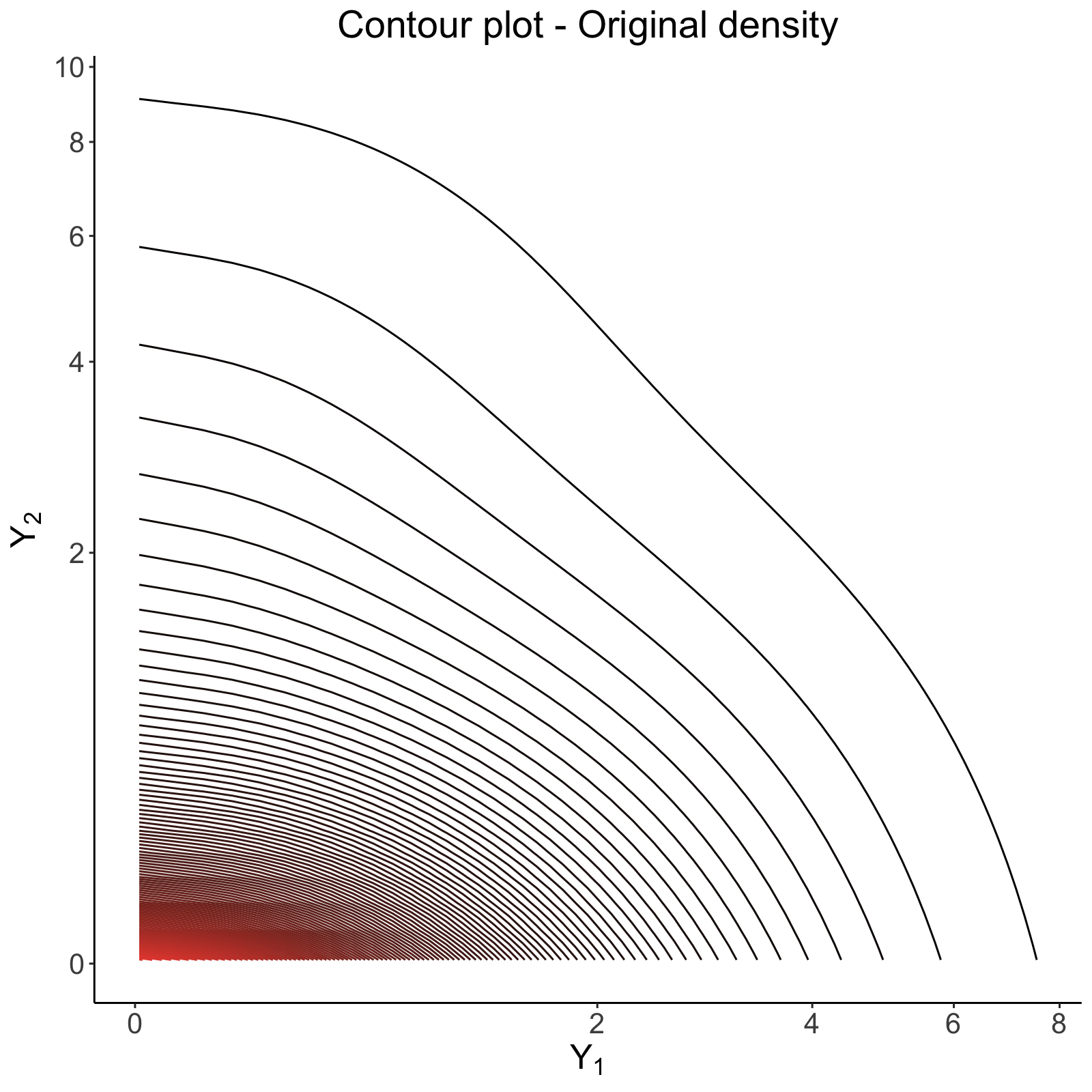}
\includegraphics[width=0.325\textwidth]{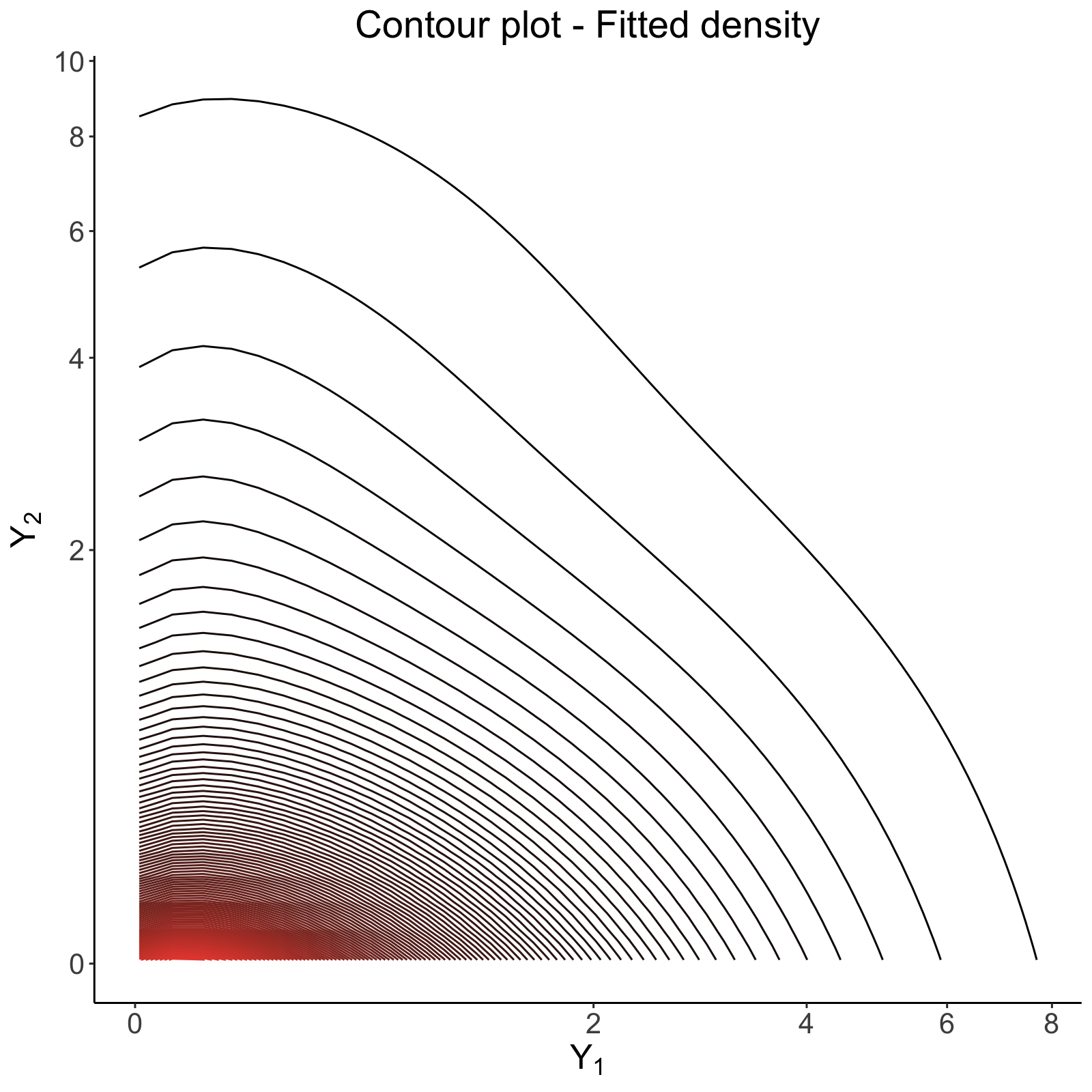}
\caption{Contour plot of the sample (left), contour plot of original distribution (center), and contour plot of fitted distribution (right).}
\label{fig:mcph_cont}
\end{figure}

\section{Conclusion}\label{sec:conclusion}
We have provided a phase-type-based model which can result in non-exponential tail behavior by introducing random and deterministic transformations. The resulting model is generally tractable in terms of matrix calculus through the Laplace transform of the random component, and thus closed-form formulas allow for statistical and probabilistic treatments, for instance, for fully explicit generalized EM algorithms. In the univariate case, the current three main ways of generating heavy-tailed phase-type distributions fall into our framework, and several new models are introduced to complement the existing suite of hidden Markov models. In the multivariate case, we obtain generalizations of well-known frailty models with fully explicit densities, contrary to other approaches of multivariate phase-type distributions in the literature (in terms of rewards or copulas). We finally show the feasibility of the statistical implementation of our models using four different examples.

Heavy-tailed phase-type distributions are statistically attractive since their interpretation in terms of an underlying evolving process is natural in many domains of application which involve processes that traverse numerous states through time, for instance, human lifetimes or legal cases. With the models and algorithms provided in this paper, we aim to provide a clearer picture of the possibilities and limitations of Markov models for practitioners that require non-standard but interpretable models.
A promising further direction of research for generating uni- and multivariate scaled phase-type distributions is to consider a general stochastic process as time-change, which for certain choices may provide fully explicit functionals and estimation procedures while remaining conceptually simple.

\textbf{Acknowledgement.} MB would like to acknowledge financial support from the Swiss National Science Foundation Project 200021\_191984.

JY would like to acknowledge financial support from the Swiss National Science Foundation Project IZHRZ0\_180549.

\textbf{Data availability statement.} The datasets generated and analyzed during the current study are available in the Zenodo repository, at \url{https://doi.org/10.5281/zenodo.5115819
}.

\textbf{Conflict of interest statement.} MB and JY declare no conflict of interest related to the current manuscript.

\bibliographystyle{abbrv}
\bibliography{Frailty_v3.bib}

\newpage

\begin{appendices}
 \section{Heavy-tailed definitions} \label{ap:def}
 
  \begin{definition} \rm 
 A distribution function $F$ on $\R_{+}=[0, \infty)$, with corresponding survival function $S = 1 - F$, is called: \
 \begin{enumerate}
	\item  {\em Regular varying} with index $\alpha \geq 0$ if 
		\begin{align*}
			\lim_{x \to \infty} \frac{S(\lambda x)}{S( x)} = \lambda^{-\alpha} 
		\end{align*}
		for all $\lambda>0$. If $\alpha = 0$, then $F$ is called {\em slowly varying}.
	
	\item {\em Weibull-type} if 
\begin{align*}
	S(x) \sim c x^{\beta} \exp (- \lambda  x ^{\tau} )\,, \quad x\to \infty \,,
\end{align*}
	for some constants $\beta \in \mathbb{R}$ and $\tau, \, \lambda , \, c > 0$. A Weibull-type distribution has heavier than exponential tail behavior if $\tau \in (0,1)$, exponential-type behavior if $\tau = 1$, and lighter than exponential otherwise.
	\item  {\em Lognormal-type} if 
\begin{align*}
	S(x) \sim c x^{\beta} (\log x)^{\xi}\exp (- \lambda (\log x )^{\gamma} )\,, \quad x\to \infty\,,
\end{align*}
	for some constants $\beta, \xi \in \R$, $\gamma > 1$ and $\lambda , \, c > 0$. Note that  in particular, the lognormal distribution is lognomal-type with  $\gamma = 2$. 
	\end{enumerate}
 \end{definition}
\end{appendices}

\end{document}